\documentclass[DIV=14,letterpaper]{scrartcl}
\pdfoutput=1

\date{}

\usepackage{tikz}
\usetikzlibrary{calc}
\usetikzlibrary{matrix}
\usepackage{amsmath,amssymb,amsfonts,amsthm,subfigure}
\usepackage{color}                    
\usepackage[pdftex]{hyperref}

\usepackage{graphicx}
\usepackage{epstopdf}
\usepackage{enumerate}
\usepackage{booktabs,multirow,threeparttable}  
\usepackage[ruled]{algorithm2e}

\allowdisplaybreaks[4]

\newcommand{\V}[1]{\mbox{\boldmath $ #1 $}}

\newcommand{\ud}{\mathrm{d}}

\newcommand{\bey}{\begin{eqnarray}}
\newcommand{\eey}{\end{eqnarray}}

\newcommand{\beq}{\begin{equation}}
\newcommand{\eeq}{\end{equation}}
\theoremstyle{plain}
\newtheorem{theorem}{\hspace{6mm}Theorem}[section]

\newtheorem{lemma}{\hspace{6mm}Lemma\,}[section]

\theoremstyle{remark}
\newtheorem{example}{\hspace{6mm}Example}[section]

\title{New finite volume element schemes based on a two-layer dual strategy
}

\author{Weizhang Huang%
\thanks{Department of Mathematics, The University of Kansas, Lawrence, KS 66045 (whuang@ku.edu).}
\and
Xiang Wang%
\thanks{Corresponding author. School of Mathematics, Jilin University, Changchun 130012, China (wxjldx@jlu.edu.cn).}
\and Xinyuan Zhang%
\thanks{Pingshan Foreign Language School, Shenzhen 518118, China and School of Mathematics, Jilin University, Changchun 130012, China (zxy1zhang@163.com).}
}

\begin{document}
\maketitle

\begin{abstract}
A two-layer dual strategy is proposed in this work to construct a new family of high-order finite volume element (FVE-2L) schemes 
that can avoid main common drawbacks of the existing high-order finite volume element (FVE) schemes.
The existing high-order FVE schemes are complicated to construct since 
the number of the dual elements in each primary element used in their construction
increases with a rate $O((k+1)^2)$, where $k$ is the order of the scheme.
Moreover, all $k$th-order FVE schemes require a higher regularity $H^{k+2}$ than the approximation theory
for the $L^2$ theory. Furthermore, all FVE schemes lose local conservation properties over boundary dual elements
when dealing with Dirichlet boundary conditions. 
The proposed FVE-2L schemes has a much simpler construction since they
have a fixed number (four) of dual elements in each primary element. They also
reduce the regularity requirement for the $L^2$ theory to $H^{k+1}$ and preserve the local conservation law on all dual elements
of the second dual layer for both flux and equation forms. Their stability and $H^1$ and $L^2$ convergence are proved.
Numerical results are presented to illustrate the convergence and conservation properties of the FVE-2L schemes.
Moreover, the condition number
of the stiffness matrix of the FVE-2L schemes for the Laplacian operator is shown to have the same growth rate
as those for the existing FVE and finite element schemes.
\end{abstract}



\noindent{\textbf{ Abbreviated title.}}
Two-layer finite volume element schemes

\section{Introduction}
The finite volume element (FVE) method \cite{AlKubaisy.2023,Bank.1987,Bi.2007,Cai.1991,Chen.2010,Chen.2015,Cui.2010,Ewing.2002,Huang.1998,Li.2000,Li.2012b,Nie.2022,Su.2023,Wang.2014}, also known as the generalized difference method, is a type of the finite volume method 
\cite{Cui.2010,Fambri.2023,Kwon.2022,Sheng.2022,Xu.2009,Yang.2022,Yang.2013,Zhang.2019,Zhang.2015,Zou.2017} that approximates the solution of partial differential equations (PDEs) in a finite element space. It inherits many advantages of both the finite element method, such as a straightforward definition of the gradient, and the finite volume method, such as the famous local conservation law.
Till now, much progress has been made
in the algorithmic development \cite{Chen.2010,Chen.2012,Huang.1998,Li.2000,Wang.2016},
stability analysis and $H^1$ estimation \cite{Chen.2012,Chen.2015,Chou.2007,Liebau.1996,Suli.1991,Xu.2009,Zhang.2015,Zhou.2020},
$L^2$ estimation \cite{Chen.2010,Chen.2002,Lin.2015,Lv.2012,Wang.2016,Yang.2023b,Zhang.2024,Zhang.2023},
and superconvergence analysis \cite{Cao.2015,Wang.2019,Wang.2021b}.
Nevertheless, there are still some open issues that have to be addressed.

Firstly, it is still complicated to construct high-order FVE schemes.
Like other finite volume (FV) schemes, FVE schemes form their approximation equations by integrating the weak formulation of the underlying
partial differential equations over dual elements. A commonly used strategy in FVE schemes is to define a dual element around each degree
of freedom for Lagrange-type schemes. Thus, for each primary element, there are $(k+1)(k+2)/2$ dual elements for a $k$th-order FVE scheme
over triangular meshes \cite{Chen.2012,Wang.2016,Xu.2009} and $(k+1)^{2}$ dual elements for a bi-$k$th-order FVE scheme
over quadrilateral meshes \cite{Lv.2012,Zhang.2015,Zhang.2023}. 
It becomes increasingly complicated and computationally burdensome to partition a primary into so many dual regions even increasing $k$ to
$3$ and $4$.

Secondly, existing high-order FVE schemes require a higher regularity ($u\in H^{k+2}$) than
what is needed for function approximation ($u\in H^{k+1}$) for the $L^2$ theory.
The optimal $L^2$ convergence rate of a FVE scheme depends on the choice of the dual strategy.
A unified $L^2$ analysis for FVE schemes on quadrilateral meshes has been provided in \cite{Lin.2015}
by establishing some numerical quadrature equivalence, and the $L^2$ result for high-order FVE schemes
on triangular meshes has been proved in \cite{Wang.2016} by proposing an orthogonality condition.
Some other $L^2$ results for high-order FVE schemes can be found in \cite{Lv.2012,Zhang.2023}.
However, all of the above $L^2$ results require $u\in H^{k+2}$ for $k$th-order ($k\geq2$) FVE schemes,
which is a higher regularity requirement than $u\in H^{k+1}$ of the approximation theory. 

Thirdly, Dirichlet boundary conditions may disrupt the conservation property on boundary dual elements in existing FVE schemes.
To illustrate this, we take the following elliptic boundary value problem (BVP)
on a bounded polygonal domain $\Omega \subset \mathbb{R}^2$ as an example,
\begin{align}
\label{eq:second_order_elliptic_problem}
\left\{
\begin{array}{rl}
-\nabla \cdot (\mathbb{D} \nabla u) &= f, \quad \mathrm{in} \, \Omega, \\
u &= 0, \quad \mathrm{on} \, \partial \Omega,
\end{array}
\right.
\end{align}
where $f \in L^2(\Omega)$, and the diffusion tensor $\mathbb{D} = (d_{ij})_{i,j=1,2}$ is bounded by
\begin{align*}
\gamma_1(\xi,\xi) \leq (\mathbb{D}\xi,\xi) \leq \gamma_2(\xi,\xi), \quad \forall \, \xi \in \mathbb{R}^2 ,
\end{align*}
and $\gamma_1$ and $\gamma_2$ are positive constants.
Integrating the first equation in (\ref{eq:second_order_elliptic_problem}) over a dual element $K^{*}$, one has the local conservation law in equation form (with discretization) given by
\begin{align}
\label{eq:conservation-law-equation}
-\iint_{K^{*}}\nabla\cdot(\mathbb{D}\nabla u_h)\,\ud x\ud y=\iint_{K^{*}}f\,\ud x\ud y,
\end{align}
or the local conservation law in flux form (with discretization after applying the divergence theorem) given by
\begin{align}
\label{eq:conservation-law-flux}
-\int_{\partial K^{*}}(\mathbb{D}\nabla u_h)\cdot \vec{n}\,\ud s=\iint_{K^{*}}f\,\ud x\ud y.
\end{align}
These local conservation laws may not be preserved by existing FVE schemes when Dirichlet boundary conditions are used.

The objective of this work is to present a new dual strategy (called a two-layer dual strategy)
to construct FVE schemes (FVE-2L) on triangular meshes that can avoid the above mentioned issues of the existing FVE schemes.
More specifically, the dual meshes of these schemes consist of the barycenter dual mesh of the linear FVE scheme
(caleld the first dual layer) and the triangulation of the primary mesh (called the second dual layer).
Thus, the FVE-2L schemes have a fixed number (four)
of dual elements on each primary triangular element regardless of the order of the scheme.
This greatly simplifies the construction of the dual mesh, and therefore FVE schemes,
and makes the implementation of the schemes
more efficiently and less burdensome. Moreover, due largely to the use of two dual layers,
it is showed that the regularity requirement for the $L^2$ theory of FVE-2L schemes is reduced
to $u\in H^{k+1}$, which is consistent with the approximation theory.
Furthermore, FVE-2L schemes preserve (\ref{eq:conservation-law-flux}) on all dual elements of the second dual layer
because the interpolation nodes corresponding to this layer are all interior nodes. Since the numerical solution $u_{h}$
is continuously differentiable on all triangular elements of the primary mesh, FVE-2L schemes preserve the local conservation
law in equation form (\ref{eq:conservation-law-equation})  as well on dual elements of the second dual layer.
As a result, the global conservation law in both flux and equation forms is preserved on the second dual layer.
FVE-2L schemes behave more or less like existing FVE schemes on the first dual layer.

The stability and $H^1$ and $L^2$ convergence of the FVE-2L schemes are analyzed in this work.
A unified framework of \cite{Chen.2012} developed for the stability of FVE schemes with
a single dual layer on triangular meshes is used for this purpose.
It is worth mentioning that the application of the framework to our current case is not trivial.
The main difficulty is that the framework requires a matrix associated with
the trial-to-test mapping to be positive definite when the underlying triangular element is equilateral.
Unfortunately, this cannot be achieved for our current situation if a single trial-to-test mapping is used
(as done in \cite{Chen.2012}). To circumvent this difficulty, we introduce a family of trial-to-test mapping
with parameters and define a minmax optimization problem for the lower bound for the minimum angle condition.
The detail of the stability analysis is given in Section~\ref{sec:Coercivity_and_boundedness}.

It is worth emphasizing that all of the existing FVE schemes use single-layer dual meshes while the FVE-2L schemes developed in this work are based on two-layer dual meshes. 

There are hybrid finite volume methods (HFVM) (e.g. see \cite{Abgrall.2023,AlKubaisy.2023,Fambri.2023,Myers.2023})
in the literature. These methods combine a finite volume method (FVM) with other numerical techniques, such as the finite element method (FEM)
and particle methods, among others. They are commonly applied to the discretization of different parts of the computational domain
or different physical quantities. The FVMs utilized within HFVM are often low-order schemes.
The FVE-2L schemes employ finite element spaces as approximation spaces, enabling the construction
of high-accuracy numerical schemes. The FVE-2L schemes can serve as one of the components in constructing an HFVM scheme.

The remainder of this paper is organized as follows. Section~\ref{sec:FVE2L_schemes} is devoted to the description of two-layer dual meshes and
the FVE-2L schemes for elliptic problems. In Section~\ref{sec:conservation_law}, the conservation properties of the FVE-2L methods are discussed.
The stability analysis of the FVE-2L schemes is given in Section~\ref{sec:Coercivity_and_boundedness}, followed by the $H^1$ and $L^2$ error analysis in Section~\ref{sec:L2 estimate}. Numerical examples are presented in Section~\ref{sec:numerical experiments} to
the conservation properties, the optimal convergence rates, and condition number for the FVE-2L schemes. Finally,
conclusions are drawn in Section~\ref{sec:conclusion}.

\section{A two-layer dual strategy and FVE-2L schemes}
\label{sec:FVE2L_schemes}

In this section we describe a two-layer dual strategy and the corresponding FVE-2L schemes.
To be specific, we focus on BVP (\ref{eq:second_order_elliptic_problem}) in this work.

We recall that FVE schemes typically form their approximation equations
by integrating the weak formulation of the underlying PDEs over dual elements
(cf. (\ref{eq:FVE-scheme0})) and define a dual element around each degree of freedom.
For each primary element, this requires $(k+1)(k+2)/2$ dual elements
for a $k$th-order FVE scheme over triangular meshes \cite{Chen.2012,Wang.2016,Xu.2009}.
It becomes increasingly complicated and computationally burdensome to partition a primary element into so many dual regions
for higher-order accuracy; see Fig.~\ref{fig:single-layered-strategy2}.
This can be avoided with the two-layer dual strategy described in this section.

We start with describing the primary and dual meshes and function spaces used in the two-layer dual strategy.

\begin{figure}[!htbp]
    \centering
    \subfigure{
    \begin{minipage}[t]{.4\textwidth}
      \centering
      {(a) Quadratic}\\
      \includegraphics*[width=100pt]{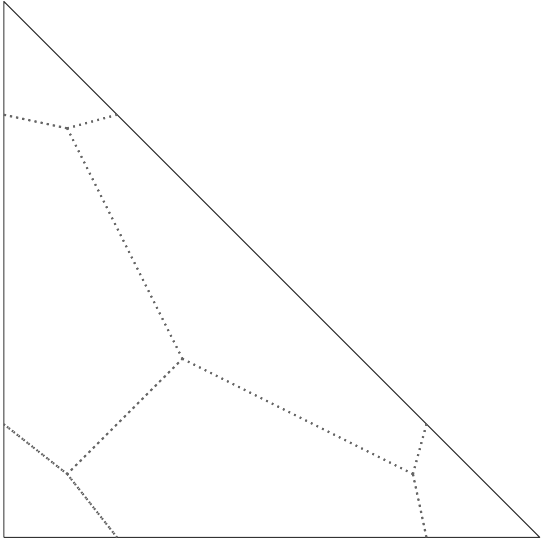}
    \end{minipage}}
    \subfigure{
    \begin{minipage}[t]{.4\textwidth}
    \centering
    {(b) Cubic}\\
      \includegraphics[width=100pt]{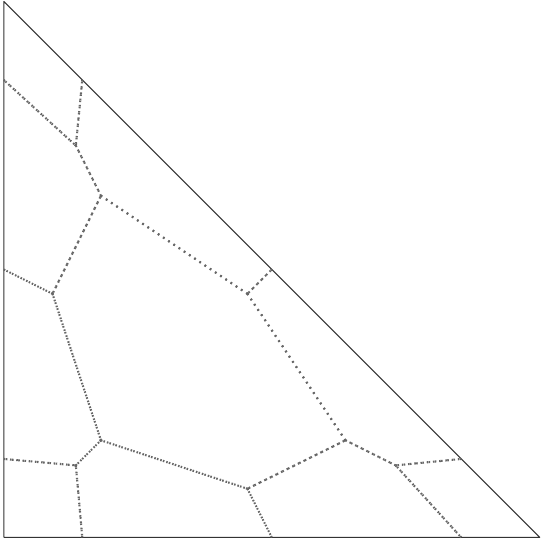}
    \end{minipage}}\\
    \subfigure{
    \begin{minipage}[t]{.4\textwidth}
      \centering
      {(c) Quartic}\\
      \includegraphics*[width=100pt]{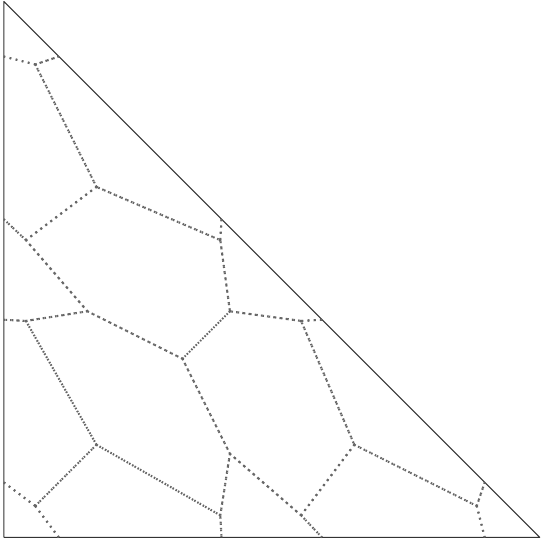}
    \end{minipage}}
    \subfigure{
    \begin{minipage}[t]{.4\textwidth}
    \centering
    {(d) Quintic}\\
      \includegraphics[width=100pt]{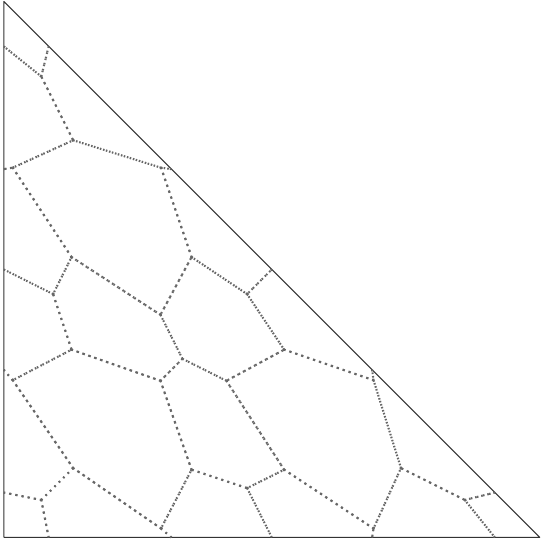}
    \end{minipage}}
    \caption{Dual elements/regions on the reference element for single-layer FVEs (cf. \cite{Wang.2016}).}
    \label{fig:single-layered-strategy2}
\end{figure}

\subsection{Primary meshes and trial spaces}
\label{subsec:primary mesh and trial space}

Let $\mathcal{T}_h = \{ K \}$ be a triangular mesh of $\Omega$, where $h$ denotes the mesh size.
The standard $k$th-order ($k\geq 2$) Lagrange finite element space over $\mathcal{T}_{h}$ is given by
\begin{align*}
U_h^{k}=\{ u_h\in C(\Omega) :\quad u_h|_{K} \in P^k(K),\quad \forall K \in \mathcal{T}_h,\quad u_h|_{\partial \Omega} =0\},
\end{align*}
where $P^k(K)$ is the set of polynomials of degree up to $k$ defined on $K$.
It is taken as the trial space for FVE-2L schemes. The exception is 
for the quadratic ($k=2$) case where there is an additional bubble function on each element, i.e.,
\begin{align*}
U_h^{2+b}=\{ u_h\in C(\Omega) :\quad u_h|_{K} \in P^2(K)\oplus \lambda_1\lambda_2\lambda_3,
\quad \forall K \in \mathcal{T}_h,\quad  u_h|_{\partial \Omega} =0\} .
\end{align*}
Here, $(\lambda_1, \lambda_2, \lambda_3)$ are the area coordinates of point $(x,y)\in K$.
The reason for this addition is that each FVE-2L scheme requires at least one degree of freedom
inside each element.
For notational simplicity, the trial space of FVE-2L schemes is written as
\begin{align*}
U_h=
\left\{
\begin{array}{ll}
U_h^{2+b}, &\text{ for }k=2, \\
U_h^{k}, &\text{ for }k > 2.
\end{array}
\right.
\end{align*}

\subsection{Dual meshes and test spaces}\label{subsec:Dual meshes and test spaces}

In contrast with the existing FVE schemes that use single-layer dual meshes, the FVE-2L schemes proposed in this work
use dual meshes of two layers. Each layer constitutes a complete partition of the domain $\Omega$.
In the following we describe dual meshes and corresponding test spaces in detail.

\begin{figure}[!htbp]
    \centering
    \subfigure{
    \begin{minipage}[t]{.47\textwidth}
      \centering
      {(a) The first dual layer}\\
      \includegraphics*[width=120pt]{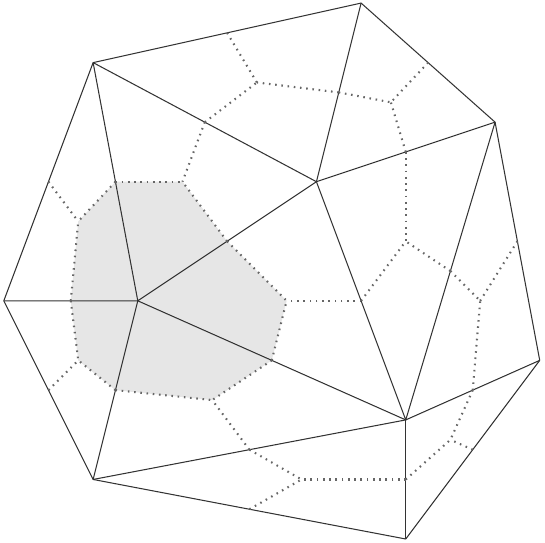}
    \end{minipage}}
    \subfigure{
    \begin{minipage}[t]{.47\textwidth}
    \centering
    {(b) The second dual layer}\\
      \includegraphics[width=120pt]{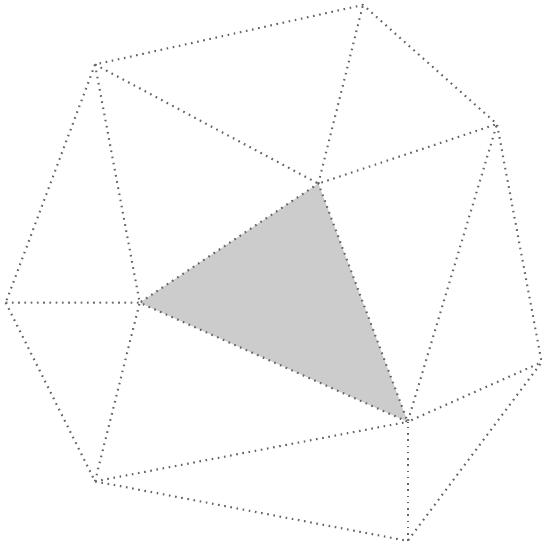}
    \end{minipage}}
    \caption{Examples of dual elements (shaded regions, $K_{\mathrm{I}}^{*}$ (left) and $K_{\mathrm{II}}^{*}$ (right)) for each dual layer.}
    \label{fig:two-layer-strategy}
\end{figure}

\textbf{Two-layer dual meshes.} The first dual layer is selected as the barycenter dual mesh of the linear FVE scheme
(see Fig.~\ref{fig:two-layer-strategy} (a)). Each dual element is composed of several quadrilaterals. These quadrilaterals correspond to
$Q_{i}$ $(i=1,\, 2,\, 3)$ on the reference element $\hat{K}=\{(x,y):\; x\geq0,\,y\geq0,\,x+y\leq1\}$,
\begin{align*}
Q_1&=\blacklozenge\{(0,0),\,(1/2,0),\,(1/3,1/3),\,(0,1/2)\},\\
Q_2&=\blacklozenge\{(1,0),\,(1/2,1/2),\,(1/3,1/3),\,(1/2,0)\},\\
Q_3&=\blacklozenge\{(0,1),\,(1,1/2),\,(1/3,1/3),\,(1/2,1/2)\}.
\end{align*}
Obviously, $Q_{1}\cup Q_{2}\cup Q_{3}=\hat{K}$. We call this barycenter dual mesh
the first dual layer and denote it by $\mathcal{T}_{\mathrm{I}}^{*}=\{K_{\mathrm{I}}^{*}\}$.

The second dual layer is taken as the primary mesh (see Fig.~\ref{fig:two-layer-strategy} (b)) and thus, each
triangle of the primary mesh serves as a dual element of the second dual layer.  We denote
\begin{align*}
Q_4& =\blacktriangle\{(0,0),\,(1,0),\,(0,1)\}=\hat{K}.
\end{align*}
This second dual layer is denoted by
$\mathcal{T}_{\mathrm{II}}^{*}=\{K_{\mathrm{II}}^{*}\}$ and the total dual mesh is denoted
by $\mathcal{T}_h^{*}=(\mathcal{T}_{\mathrm{I}}^{*};\,\mathcal{T}_{\mathrm{II}}^{*})$.
Hereafter, $K^{*}\in\mathcal{T}_h^{*}$ is used to denote a general dual element $K^{*}$
in either $\mathcal{T}_{I}^{*}$ or $\mathcal{T}_{II}^{*}$. Notice that $K^{*}$ is either a polygon or a triangle.

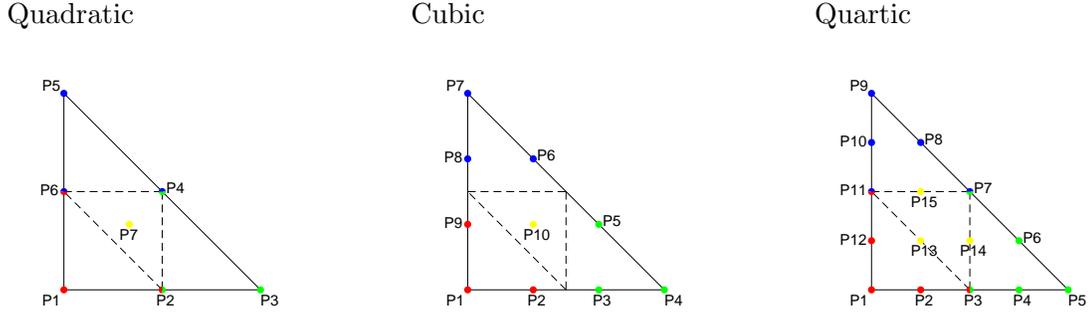
\begin{figure}[!htbp]
    \centering
    \subfigure{
    \begin{picture}(-3,-3)
    \put(30,130){Quadratic}
    \end{picture}
    \begin{minipage}[t]{.30\textwidth}
      \centering
      \includegraphics*[width=170pt]{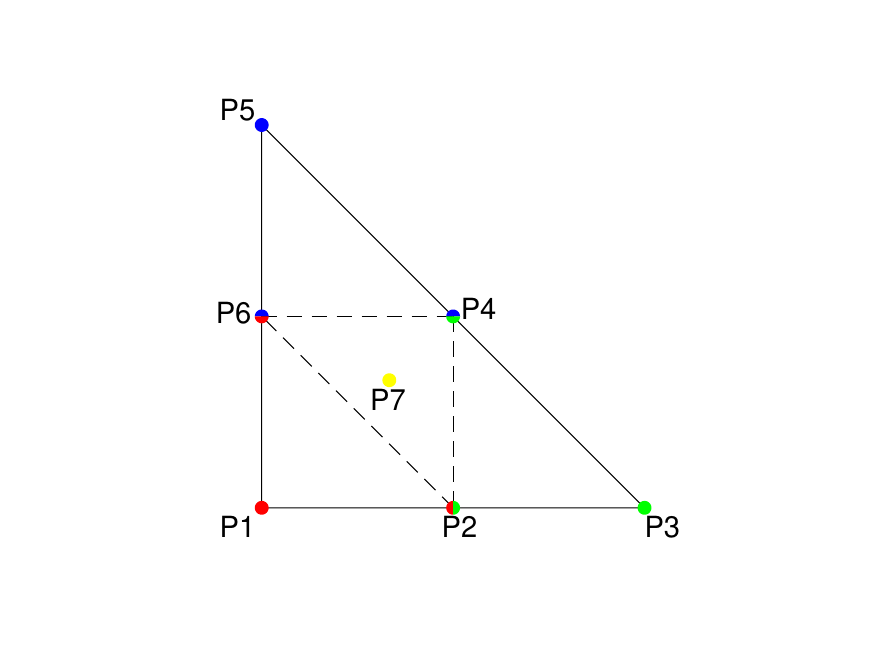}
    \end{minipage}}
    \subfigure{
    \begin{picture}(-3,-3)
    \put(30,130){Cubic}
    \end{picture}
    \begin{minipage}[t]{.30\textwidth}
    \centering
      \includegraphics[width=170pt]{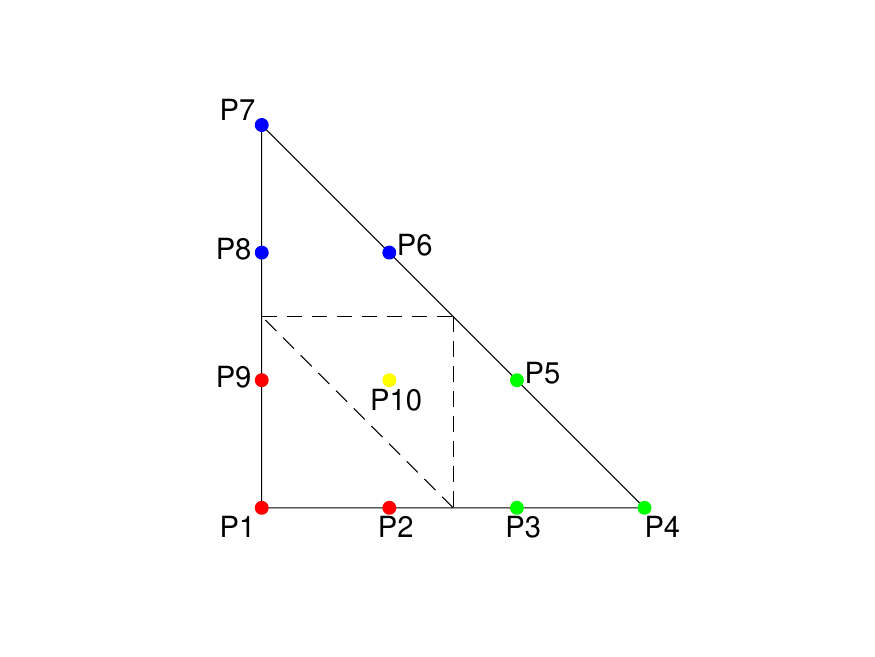}
    \end{minipage}}
    \subfigure{
    \begin{picture}(-3,-3)
    \put(30,130){Quartic}
    \end{picture}
    \begin{minipage}[t]{.30\textwidth}
    \centering
      \includegraphics[width=170pt]{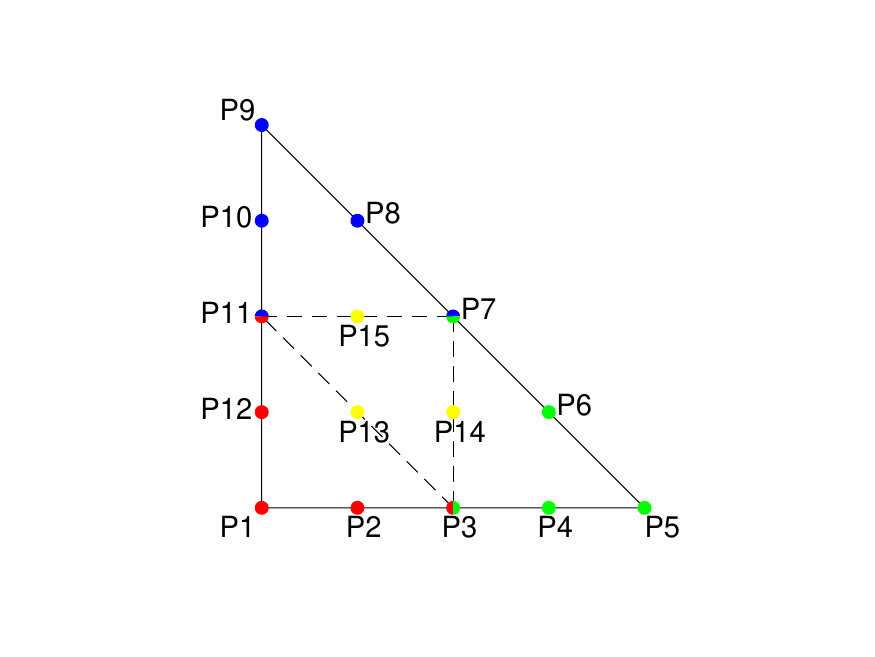}
    \end{minipage}}\\
    \caption{Interpolation nodes on the reference element $\hat{K}$. The colored points refer to the interpolation nodes on $Q_1$ (red), $Q_2$ (green), $Q_3$ (blue), and $Q_4$ (yellow). Points with two colors are shared freedoms.}
    \label{fig:nodes-4-FVE-2L-schemes}
\end{figure}

\textbf{Test spaces.}
For each FVE-2L scheme, there are two test function spaces corresponding to the two layers of the dual mesh.
The degrees of freedom and dual regions on the reference element $\hat{K}$ are showed
in Fig.~\ref{fig:nodes-4-FVE-2L-schemes} and Table~\ref{tab:nodes-function-space} for the cases of $k=2,\,3,\,4$.
The analytical expressions of the test basis functions are given in Appendix~\ref{SEC:test-basis-functions}.
We refer to Fig.~\ref{fig:nodes-4-FVE-2L-schemes} and Table~\ref{tab:nodes-function-space}
for the notation of interpolation nodes, dual regions $Q_{i}$ ($i=1,...,4$), sets of interpolation nodes thereon,
and test spaces.

Next, we provide some detailed explanations of the test spaces.

\begin{table}[htbp!]
\centering
  \caption{The interpolation nodes and function spaces. $P^{2-\lambda\lambda}(Q_{i})$  is the incomplete
  quadratic polynomial space on $Q_{i}$ that excludes the quadratic basis functions $\lambda_{i_{1}}\lambda_{i_{2}}$
  ($i_{1},i_{2}\in\{1,2,3\}\setminus\{i\}$, $i_{1}<i_{2}$).}
\label{tab:nodes-function-space}
  \begin{tabular}{ c | c | c | l | c }
  \toprule
  \hline
  FVE-2L schemes & dual layer  & area & interpolation nodes & test function spaces\\
  \hline
  \multirow{4}*{quadratic}&\multirow{3}*{first dual layer}&$Q_{1}$&$\mathcal{N}_{1}=\{P_1,P_2,P_6\}$&\multirow{3}*{$P^1(Q_{i})$ $(i=1,2,3)$}\\
  \cline{3-4}
  &&$Q_{2}$&$\mathcal{N}_{2}=\{P_2,P_3,P_4\}$\\
  \cline{3-4}
  &&$Q_{3}$&$\mathcal{N}_{3}=\{P_4,P_5,P_6\}$\\
  \cline{2-5}
  &second dual layer&$Q_{4}$&$\mathcal{N}_{4}=\{P_7\}$&$P^0(Q_{4})$\\
  \hline
  \multirow{4}*{cubic}&\multirow{3}*{first dual layer}&$Q_{1}$&$\mathcal{N}_{1}=\{P_1,P_2,P_9\}$&\multirow{3}*{$P^1(Q_{i})$ $(i=1,2,3)$}\\
  \cline{3-4}
  &&$Q_{2}$&$\mathcal{N}_{2}=\{P_3,P_4,P_5\}$\\
  \cline{3-4}
  &&$Q_{3}$&$\mathcal{N}_{3}=\{P_6,P_7,P_8\}$\\
  \cline{2-5}
  &second dual layer&$Q_{4}$&$\mathcal{N}_{4}=\{P_{10}\}$&$P^0(Q_{4})$ \\
  \hline
  \multirow{4}*{quartic}&\multirow{3}*{first dual layer}&$Q_{1}$&$\mathcal{N}_{1}=\{P_1,P_2,P_3,P_{11},P_{12}\}$   
            &\multirow{3}*{      $P^{2-\lambda\lambda}(Q_{i})$   $(i=1,2,3)$}\\
  \cline{3-4}
  &&$Q_{2}$&$\mathcal{N}_{2}=\{P_3,P_4,P_5,P_6,P_7\}$\\
  \cline{3-4}
  &&$Q_{3}$&$\mathcal{N}_{3}=\{P_7,P_8,P_9,P_{10},P_{11}\}$\\
  \cline{2-5}
  &second dual layer&$Q_{4}$&$\mathcal{N}_{4}=\{P_{13},P_{14},P_{15}\}$&$P^1(Q_{4})$\\
  \hline
  \bottomrule
  \end{tabular}
\end{table}

Each interpolation node on the first layer of the dual mesh corresponds to a test function for the first dual layer.
Such a test function has the support including the quadrilaterals sharing the node and is defined
separately on each of these quadrilaterals. Consider the test space on the reference element $\hat{K}$.
For $i\in \{ 1,\, 2,\, 3\}$, a test basis function $\hat{\psi}_j$ associated with $P_{j}\in \mathcal{N}_{i}$
and restricted on $Q_{i}$ is the Lagrange interpolation polynomial satisfying
\begin{align}
\label{eq:test function}
\hat{\psi}_j(P_s)&= \delta_{j,s}, 
\quad \forall P_{s}\in \mathcal{N}_{i} .
\end{align}
Notice that $\hat{\psi}_j|_{Q_{i}}$ belongs to $P^{1}(Q_{i})$ for both quadratic and cubic FVE-2L schemes and $P^{2-\lambda\lambda}(Q_{i})$ for the quartic FVE-2L scheme. 
Moreover, $P^{2-\lambda\lambda}(Q_{i}) =\mathrm{Span} \{\lambda_{1},\,\lambda_{2},\,\lambda_{3},\,\lambda_{i}\lambda_{i_{1}},\,\lambda_{i}\lambda_{i_{2}}\}$ ($i_{1},i_{2}\in\{1,2,3\}\setminus\{i\}$, $i_{1}<i_{2}$) is an incomplete quadratic polynomial space on $Q_{i}$
 that excludes quadratic basis functions $\lambda_{i_{1}}\lambda_{i_{2}}$.
Thus, the test space on $\hat{K}$ for the first dual layer is given by
\begin{align}   \label{eq:VKI}
V_{\hat{K},\mathrm{I}}&=\{v_{\hat{K},\mathrm{I}}:\quad v_{\hat{K},\mathrm{I}}=\sum_{P_j\in\mathcal{N}_{1}\cup\mathcal{N}_{2}\cup\mathcal{N}_{3}} v_j\hat{\psi}_j\}.
\end{align}
The degrees of freedom of $V_{\hat{K},\mathrm{I}}$ is \#$(\mathcal{N}_{1}\cup\mathcal{N}_{2}\cup\mathcal{N}_{3})=3k$ for the $k$th-order FVE-2L scheme.

From Fig.~\ref{fig:nodes-4-FVE-2L-schemes} and Table~\ref{tab:nodes-function-space}, we can see that
there are some shared degrees of freedom between different dual quadrilaterals (first dual layer) for quadratic and quartic FVE-2L schemes.
In this case, the test basis functions are defined on each dual quadrilateral separately, and they are continuous
at the shared midline. For example, $\hat{\psi}_2$ for the quadratic FVE-2L scheme is continuous on the segment $\overline{P_{2}P_{7}}$;
cf. Fig.~\ref{fig:nodes-4-FVE-2L-schemes}. From this perspective, FVE-2L schemes are different from the existing FVE schemes and
do not pursue completely independent (discontinuous) test functions on each dual element of the first dual layer.

A test function for the second dual layer is a polynomial on each dual element $K_{\mathrm{II}}^{*}\in\mathcal{T}_{\mathrm{II}}^{*}$.
For $P_{j}\in \mathcal{N}_{4}$, $\hat{\psi}_j$ is defined as the Lagrange interpolation polynomial on $Q_{4}$ such that
\begin{align}
\label{eq:test function4}
\hat{\psi}_j(P_s)&= \delta_{j,s}, 
\quad \forall P_{s}\in \mathcal{N}_{4}.
\end{align}
Notice that  $\hat{\psi}_j|_{Q_{4}}$ belongs to $P^{0}(Q_{4})$ for both the quadratic and cubic FVE-2L schemes
and $P^{1}(Q_{4})$ for the quartic FVE-2L scheme. Thus, the test space on $\hat{K}$ for the second dual layer is given by
\begin{align}
\label{eq:VKII}
V_{\hat{K},\mathrm{II}}&=\{v_{\hat{K},\mathrm{II}}:\quad v_{\hat{K},\mathrm{II}}=\sum_{P_j\in\mathcal{N}_{4}} v_j\hat{\psi}_j\}.
\end{align}
The degrees of freedom of $V_{\hat{K},\mathrm{II}}$ is \#$\mathcal{N}_{4}$, which is 1 for the quadratic and cubic FVE-2L schemes
and 3 for the quartic FVE-2L scheme.
Note that the number of the degrees of freedom of $V_{\hat{K},\mathrm{I}}$ plus those of $V_{\hat{K},\mathrm{II}}$
marches  the number of the degrees of freedom of the trial function space on $\hat{K}$.

The test function spaces on $K$, $V_{K,\mathrm{I}}$ and $V_{K,\mathrm{II}}$, can be obtained from (\ref{eq:VKI}) and (\ref{eq:VKII})
through the affine mapping from $\hat{K}$ to $K$.
The test spaces for the first and second dual layers of the dual mesh can be denoted by
\begin{align*}
V_{\mathrm{I}}&=\{v_{\mathrm{I}}:\,v_{\mathrm{I}}|_{K}\in V_{K,\mathrm{I}}\,\,\,\forall K\in\mathcal{T}_{h}
                    \quad \mathrm{and} \quad  v_{\mathrm{I}}|_{K_{\mathrm{I}}^{*}}\in C(K_{\mathrm{I}}^{*})\,\,\,\forall K_{\mathrm{I}}^{*}\in\mathcal{T}_{\mathrm{I}}^{*} \},    \\
V_{\mathrm{II}}&=\{v_{\mathrm{II}}:\, v_{\mathrm{II}}|_{K}\in V_{K,\mathrm{II}}\,\,\,\forall K\in\mathcal{T}_{h} \}.
\end{align*}
For convenience, we also use the notation
\begin{align*}
V_{h}=\{(v_{\mathrm{I}};\,v_{\mathrm{II}}):\,v_{\mathrm{I}}\in V_{\mathrm{I}},\,\,v_{\mathrm{II}}\in V_{\mathrm{II}}\}.
\end{align*}

\subsection{High-order FVE-2L schemes}

A FVE-2L scheme can be obtained by multiplying PDE (\ref{eq:second_order_elliptic_problem})
with a test function, integrating it on a dual element, and applying the divergence theorem (integration by parts).
The scheme can be expressed as finding $u_h\in U_h$ such that
\begin{align}
\label{eq:FVE-scheme0}
\left\{
\begin{array}{rll}
a^{*}_{\mathrm{I}}(u_h,\,v_{\mathrm{I}}) &= (f,\,v_{\mathrm{I}}),    &\forall v_{\mathrm{I}}\in V_{\mathrm{I}},\\
a^{*}_{\mathrm{II}}(u_h,\,v_{\mathrm{II}})&=(f,\,v_{\mathrm{II}}),   &\forall v_{\mathrm{II}}\in V_{\mathrm{II}},
\end{array}
\right.
\end{align}
where
\begin{align*}
a^{*}_{\mathrm{I}}(u_h,\,v_{\mathrm{I}})
    &
            =\sum_{K_{\mathrm{I}}^{*}\in \mathcal{T}_{\mathrm{I}}^{*}}(\iint_{K_{\mathrm{I}}^{*}}(\mathbb{D}\nabla u_h)\cdot\nabla v_{\mathrm{I}}\,\ud x \ud y-\int_{\partial K_{\mathrm{I}}^{*} }(\mathbb{D}\nabla u_h)\cdot\vec{n}\,v_{\mathrm{I}}\,\ud s), 
       \\
    &=\sum_{K\in\mathcal{T}_h}\sum_{K_{\mathrm{I}}^{*}\in \mathcal{T}_{\mathrm{I}}^{*}}(\iint_{K_{\mathrm{I}}^{*}\cap K}(\mathbb{D}\nabla u_h)\cdot\nabla v_{\mathrm{I}}\,\ud x \ud y-\int_{\partial K_{\mathrm{I}}^{*} \cap K}(\mathbb{D}\nabla u_h)\cdot\vec{n}\,v_{\mathrm{I}}\,\ud s),   \\
a^{*}_{\mathrm{II}}(u_h,\,v_{\mathrm{II}})
    & =\sum_{K_{\mathrm{II}}^{*}\in \mathcal{T}_{\mathrm{II}}^{*}}(\iint_{K_{\mathrm{II}}^{*}}(\mathbb{D}\nabla u_h)\cdot\nabla v_{\mathrm{II}}\,\ud x \ud y-\int_{\partial K_{\mathrm{II}}^{*}}(\mathbb{D}\nabla u_h)\cdot\vec{n}\,v_{\mathrm{II}}\,\ud s).
\end{align*}
This can be written more compactly as
\begin{align}\label{eq:FVE-scheme}
a^{*}(u_h,\,v_h)=b^{*}(u_h,\,v_h),\quad \forall v_h =(v_{\mathrm{I}};\,v_{\mathrm{II}})\in V_h,
\end{align}
where
\begin{align*}
a^{*}(u_h,\,v_h) = a^{*}_{\mathrm{I}}(u_h,\,v_{\mathrm{I}})+a^{*}_{\mathrm{II}}(u_h,\,v_{\mathrm{II}}),\quad
b^{*}(f,\,v_h)=(f,\,v_{\mathrm{I}})+(f,\,v_{\mathrm{II}}).
\end{align*}

\section{Conservation laws}
\label{sec:conservation_law}

Conservation laws are fundamental physical properties and it is highly desired to preserve them in numerical discretizations.
Generally speaking, the finite volume (element) method is well known for its preservation of the local conservation law
in flux form (\ref{eq:conservation-law-flux}). However, this local conservation law on boundary dual elements is violated
in the existing FVE schemes when Dirichlet boundary conditions are used.
Moreover, to the authors' best knowledge, the local conservation law in equation form (\ref{eq:conservation-law-equation})
has not been addressed in literature so far. In this section we show that FVE-2L schemes
preserve the local and global conservation law in both flux and equation forms on the second dual layer with the help of the two-layer dual strategy.

Recall that any $v_{\mathrm{II}}\in V_{\mathrm{II}}$ is a piecewise polynomial on the second dual layer $\mathcal{T}_{\mathrm{II}}^{*}$. Taking $v_{\mathrm{II}}$ as the characteristic function of $K_{\mathrm{II}}^{*}\in\mathcal{T}_{\mathrm{II}}^{*}$
in the second equation of (\ref{eq:FVE-scheme0}), one has
\begin{align}
\label{eq:conservation-law-flux-for-FVE-2L}
-\int_{\partial K_{\mathrm{II}}^{*}}(\mathbb{D}\nabla u_h)\cdot\vec{n}\,\ud s=\iint_{K_{\mathrm{II}}^{*}}f\,\ud x \ud y,
\end{align}
which is the local conservation law in flux form (\ref{eq:conservation-law-flux}) on $K_{\mathrm{II}}^{*}$.
Since all dual elements of the second dual layer $\mathcal{T}_{\mathrm{II}}^{*}$ correspond to interior computing nodes,
Dirichlet boundary conditions have no effect on the conservation laws for dual elements of the second dual layer.
The global conservation law in flux form follows from the fact that $\mathcal{T}_{\mathrm{II}}^{*}$ forms a complete partition
of the domain $\Omega$ by itself.

By definition,  $K_{\mathrm{II}}^{*} \in \mathcal{T}_{II}^{*}$ is a triangular element of the primary mesh and
$u_h$ is continuously differentiable on $K_{\mathrm{II}}^{*}$. Applying the divergence theorem
to (\ref{eq:conservation-law-flux-for-FVE-2L}), one has
\begin{align}
\label{eq:conservation-law-equa-for-FVE-2L}
-\iint_{K_{\mathrm{II}}^{*}}\nabla\cdot(\mathbb{D}\nabla u_h)\,\ud x \ud y = \iint_{K_{\mathrm{II}}^{*}}f\,\ud x \ud y.
\end{align}
Thus, we have the local conservation law in equation form (\ref{eq:conservation-law-equation}) on $K_{\mathrm{II}}^{*}$ for FVE-2L schemes. The global conservation law in equation form follows from the summation over all $K_{\mathrm{II}}^{*}\in \mathcal{T}_{\mathrm{II}}^{*}$.

On the other hand, FVE-2L schemes behave more or less like existing FVE schemes on the first dual layer
(cf. numerical examples in Section~\ref{sec:numerical experiments}).
Generally speaking, they do not preserve the local conservation law in both flux and equation forms on dual elements
of the first dual layer. The exception is odd-order FVE-2L schemes that preserve the local conservation law
in flux form on the interior elements of the first dual layer (or all elements if no Dirichlet boundary conditions are used)
due to the independence of their test spaces.

FVE-2L schemes preserve the global conservation law in equation form but not in flux form for the first dual layer. 
The former is a consequence of the equality
\begin{align*}
\sum_{K_{I}^{*}\in\mathcal{T}_{I}^{*}}\iint_{K_{I}^{*}}\nabla\cdot(\mathbb{D}\nabla u_h)\,\ud x\ud y
=\sum_{K_{II}^{*}\in\mathcal{T}_{II}^{*}}\iint_{K_{II}^{*}}\nabla\cdot(\mathbb{D}\nabla u_h)\,\ud x\ud y.
\end{align*}

As will be seen from the error analysis in Section~\ref{sec:L2 estimate}, FVE-2L schemes are convergent. As a consequence,
the magnitude of the difference between two sides of (\ref{eq:conservation-law-equation}) or (\ref{eq:conservation-law-flux}),
an indicator of the severity of the violation of the conservation law, decreases as the mesh is being refined.

\section{Stability and boundedness}
\label{sec:Coercivity_and_boundedness}

\subsection{Stability for existing FVE schemes}

The FVE method is a Petrov-Galerkin method where the trial and test spaces are selected differently.
Its stability is ensured by the inf-sup condition
\begin{align}
\inf_{u_{h}\in U_{h}} \sup_{v_{h}\in V_{h}} \frac{a(u_{h},\,v_{h})}{|u_{h}|_{1}\,|v_{h}|_{1,\mathcal{T}_{h}^{*}}}\geq C,
\end{align}
where $U_{h}$, $V_{h}$, and $a(\cdot,\,\cdot)$ are the trial space, test space, and bilinear form of the underlying FVE scheme,
respectively, $|\cdot |_{1}$ denotes the $H^1$ semi-norm, and $C$ is a positive constant independent of $u_{h}$ and $h$.
In the context of FVE methods, the inf-sup condition is typically derived
by defining a trial-to-test mapping $\Pi^{*}$ and proving the so-called uniform local-ellipticity condition
\begin{align}
\label{eq:stability_FVE_one}
a_{K}(u_{h},\,\Pi^{*}u_{h})\geq C_{1}|u_{h}|_{1,K}^{2},\quad \forall K \in \mathcal{T}_h,
\end{align}
where $a_{K}(\cdot,\,\cdot)$ is the bilinear form on $K$;
e.g., see \cite{Chen.2012,Lin.2015,Lv.2012,Wang.2016,Zhang.2023,Zhang.2015,Zou.2017}. 
For triangular meshes, this condition has been derived (e.g., see \cite{Chen.2012,Xu.2009,Zhou.2020})
from the minimum angle condition requiring that all of the interior angles of the triangular elements
be greater than or equal to a certain positive lower bound.

A framework for computing the lower bound for the minimum angle condition was developed
in \cite{Chen.2012} for high-order FVE schemes on triangular meshes.
Define
\begin{align}
\textbf{H}(r_1,r_2)=I+ \tilde{\textbf{A}}_0+r_1\tilde{\textbf{A}}_1
+r_2\tilde{\textbf{A}}_2,
\label{H-0}
\end{align}
where $r_1 = \frac{|l_1|^2}{|l_0|^2}$ and $r_2 = \frac{|l_2|^2}{|l_0|^2}$,
$l_0$, $l_1$, and $l_2$ represent the three edges of an arbitrary triangular element,
$I$ is the identity matrix, $\tilde{\textbf{A}} = (\textbf{A} + \textbf{A}^{T})/2$, and
\begin{align*}
\textbf{A}_0& =\textbf{A}_{0,1}+\textbf{A}_{0,2}-\textbf{A}_{1,2},\\
\textbf{A}_1& =\textbf{A}_{0,1}-\textbf{A}_{0,2}+\textbf{A}_{1,2},\\
\textbf{A}_2& =-\textbf{A}_{0,1}+\textbf{A}_{0,2}+\textbf{A}_{1,2},
\\
\textbf{A}_{0,1} & =[\hat{a}_{0,1}(\hat{\phi_{j}},\hat{\psi_i})],
\quad \textbf{A}_{0,2} =[\hat{a}_{0,2}(\hat{\phi_{j}},\hat{\psi_i})],\quad  \textbf{A}_{1,2}=[\hat{a}_{1,2}(\hat{\phi_{j}},\hat{\psi_i})],
\\
\hat{a}_{0,1}(\hat{\phi},\hat{\psi})
    =& \sum_{Q\in\mathrm{supp}(\hat{\psi})\cap \hat{K}} 
        (\iint_{Q}  \frac{\partial \hat{\phi}}{\partial x}\frac{\partial\hat{\psi}}{\partial x}\,\ud x\ud y
        -\int_{\partial Q\cap\hat{K}^{\circ}}\hat{\psi}\frac{\partial\hat{\phi}}{\partial x}\,\ud y),\\
\hat{a}_{0,2}(\hat{\phi},\hat{\psi})
    =& \sum_{Q\in\mathrm{supp}(\hat{\psi})\cap \hat{K} } 
        (\iint_{Q}  \frac{\partial \hat{\phi}}{\partial y}\frac{\partial\hat{\psi}}{\partial y}\,\ud x\ud y+\int_{\partial Q \cap\hat{K}^{\circ}}\hat{\psi}\frac{\partial\hat{\phi}}{\partial y}\,\ud x),\\
\hat{a}_{1,2}(\hat{\phi},\hat{\psi})
    =& \sum_{Q\in\mathrm{supp}(\hat{\psi})\cap \hat{K}} 
        \iint_{Q}  (\frac{\partial \hat{\phi}}{\partial x}-\frac{\partial \hat{\phi}}{\partial y})(\frac{\partial\hat{\psi}}{\partial x}-\frac{\partial\hat{\psi}}{\partial y})\,\ud x\ud y\\
    &-\sum_{Q\in\mathrm{supp}(\hat{\psi})\cap \hat{K}}
        \int_{\partial Q\cap\hat{K}^{\circ}}\hat{\psi}(\frac{\partial\hat{\phi}}{\partial x}-\frac{\partial\hat{\phi}}{\partial y})\,(\ud y+\ud x).
\end{align*}
Here, $\hat{\phi}$ and $\hat{\psi}$ are trial and test basis functions on $\hat{K}$ and $\hat{K}^{\circ}$
denotes the interior of $\hat{K}$.
It was shown in \cite{Chen.2012} that a lower bound for a sufficient minimum angle condition is given by
\begin{align}
\mathcal{B} =\sup\limits_{(r_1,r_2)\in \Gamma_\textbf{H}}\theta_{\min}(r_1,r_2),
\label{optimization-B1-0}
\end{align}
where
\begin{align*}
& \theta_{\min}(r_1,r_2) = \arccos\left (\frac{1+r_1-r_2}{2(r_1)^{1/2}}\right ),
\\
& \Gamma_\textbf{H} =\{(r_1,r_2):\;\underline{r_1}<r_1\leq 1, \; r_2 = \underline{r_2}(r_1)\},
\\
& \underline{r_1} =1-\frac{1}{\lambda_{\max}(\textbf{H}(1,1),\tilde{\textbf{A}}_1+\tilde{\textbf{A}}_2)},
\\
& \underline{r_2}(r_1) =r_1-\frac{1}{\lambda_{\max}(\textbf{H}(r_1,r_1),\tilde{\textbf{A}}_2)},\qquad r_1\in(\underline{r_1},1],
\end{align*}
and $\lambda_{\max}(\textbf{B},\textbf{A})$ denotes the maximum generalized eigenvalue of $\textbf{A}$ with
respect to $\textbf{B}$, i.e. the largest root of the algebraic equation
$\det (\textbf{A}-\lambda \textbf{B})=0$. Notice that the above optimization is well defined only when
$\textbf{H}(r_1,r_1)$, for $r_1\in(\underline{r_1},1]$, and particularly, $\textbf{H}(1,1)$, is positive definite,
which was shown in \cite{Chen.2012} to be true for the existing FVE schemes.
Notice also that $r_1 = r_2 = 1$ implies that the underlying triangular element is equilateral.

An upper bound that is more economic to compute was suggested in \cite{Chen.2012}. For a given positive integer $N$,
define $\Gamma_{\textbf{H},N} =\bigcup\limits_{k=0}^{N-1}\overline{R_kR_{k+1}}$, where
\[
R_k =(r_1^{[k]},r_2^{[k]})\in \Gamma_\textbf{H}, \quad
r_1^{[k]} = \underline{r_1}+\frac{k(1-\underline{r_1})}{N}, \quad
r_2^{[k]} = \underline{r_2}(r_1^{[k]}), \quad k=0,1,\cdots,N .
\]
The upper bound is defined as
\[
\mathcal{B}_{N} = \sup_{(r_1,r_2)\in\Gamma_{\textbf{H},N}}\theta_{\min}(r_1,r_2).
\]
It was shown in \cite{Chen.2012} that the above optimization problem is equivalent mathematically to
\begin{equation}
\mathcal{B}_{N} =\max\limits_{0\leq k\leq N-1}\overline{\theta_{\min}}(r_1^{[k]},r_2^{[k]},r_1^{[k+1]},r_2^{[k+1]}),
\label{optimization-B1-1}
\end{equation}
where
\begin{align*}
\overline{\theta_{\min}}(s_1,s_2,t_1,t_2) =
\left\{
\begin{array}{ll}
\arccos c_{ST}(t_1),&g=0, \\
\arccos\min\{c_{ST}(s_1),c_{ST}(t_1)\}, &g\neq 0\,\, \mathrm{and}\,\, \frac{h}{g}\notin[s_1,t_1],\\
\arccos\min\{c_{ST}(s_1),c_{ST}(t_1),c_{ST}(h/g)\},&\,\mathrm{otherwise},
\end{array}
\right.
\end{align*}
and
\begin{align*}
g =1-\frac{t_2-s_2}{t_1-s_1},\quad  h =1+s_1-s_2-gs_1,\quad  c_{ST}(r) =\frac{gr+h}{2r^{1/2}}.
\end{align*}

\subsection{Stability for FVE-2L schemes}

We want to apply the framework of \cite{Chen.2012}
(described in the previous subsection) to the study of the stability for FVE-2L schemes.
This application is not trivial.
The main difficulty is that the matrix defined in (\ref{H-0}) is not positive definite for $r_1 = r_2 = 1$
and therefore the optimization problem (\ref{optimization-B1-0}) is not well defined.
To circumvent this difficulty, we introduce a family of trial-to-test mappings
with parameters and define a constrained minmax optimization problem over $r_1$ and $r_2$ and the parameters
involved in the trial-to-test mappings for the lower bound for the minimum angle condition.

Recall that, for the reference element $\hat{K}$,
$\hat{u}_h \in \hat{U}_h^k$ can be expressed as $\hat{u}_h = \sum_{i=1}^{N_K} \hat{u}_i \hat{\phi}_i$
and $\hat{v}_h \in \hat{V}_h$ can be expressed as $\hat{v}_h = \sum_{i=1}^{N_K} \hat{v}_i \hat{\psi}_i$.
Then, we define the parametric trial-to-test mapping $\hat{\Pi}_{\textbf{a},\textbf{b}}^{*}: \hat{U}_h \to \hat{V}_h$ as
\begin{align}
\label{eq:trial-to-test-mapping-def}
\left(
\begin{array}{c}
\hat{v}_1\\
\vdots\\
\hat{v}_{N_{K}}
\end{array}
\right)
=
M_k(\mathbf{a},\mathbf{b})
\left(
\begin{array}{c}
\hat{u}_1\\
\vdots\\
\hat{u}_{N_{K}}
\end{array}
\right),
\end{align}
where $\V{a}$ and $\V{b}$ are the parameters and
$M_k(\mathbf{a},\mathbf{b})$ is a matrix of size $N_{K} \times N_{K}$ and its expression is
given in Appendix~\ref{Appendix:trial-to-test-mapping}.
The trial-to-test mapping on any triangular element $K$ can be obtained from $\hat{\Pi}_{\textbf{a},\textbf{b}}^{*}$
through the affine mapping between $\hat{K}$ and $K$.
For this family of trial-to-test mappings we define
\begin{align}
\textbf{H}(r_1,r_2,\mathbf{a},\mathbf{b})=I+ M_k(\mathbf{a},\mathbf{b}) \left ( \tilde{\textbf{A}}_0+r_1\tilde{\textbf{A}}_1
+r_2\tilde{\textbf{A}}_2 \right ),
\label{H}
\end{align}
With this definition of $\textbf{H}$, the maximization (\ref{optimization-B1-1}) can be performed
if $\textbf{H}(1,1,\mathbf{a},\mathbf{b})$ is semi-definite. Moreover, the maximum value is a function of
the parameters $\mathbf{a}$ and $\mathbf{b}$, i.e., $\mathcal{B}_{N} = \mathcal{B}_{N}(\mathbf{a},\mathbf{b})$.
This function is minimized over the parameters under the constraint that
$\textbf{H}(1,1,\mathbf{a},\mathbf{b})$ is semi-definite, viz.,
\begin{align}
\label{minimum_problem}
\begin{array}{rl}
(\textbf{a}^{*},\textbf{b}^{*})=
    \arg \min &  \mathcal{B}_{N}(\textbf{a},\textbf{b}),\\
    \mathrm{s.t.} &       \textbf{H}(1,1,\mathbf{a},\mathbf{b})\geq 0,
\end{array}
\end{align}
where $\geq$ is in the sense of positive semi-definiteness.
This problem is highly nonlinear and needs to be solved numerically.
The computed lower bound $\mathcal{B}_N$ for the minimum angle and the corresponding optimal values for $\V{a}$ and $\V{b}$
for FVE-2L schemes are reported in Table~\ref{tab:minimum-angle}.

According to the framework of \cite{Chen.2012}, the solution to the minmax problem (\ref{minimum_problem}) provides
a lower bound for the minimum angle condition.

\begin{table}[htbp!]
\centering
  \caption{The lower bound of the minimum angle ($\mathcal{B}_N$) and optimal parameters $\V{a}^*$ and $\V{b}^*$
  for $k$th-order FVE-2L schemes.}
  \label{tab:minimum-angle}
  \begin{tabular}{ c | c | c}
  \toprule
  scheme &  interpolation parameters $(\textbf{a}^{*},\textbf{b}^{*})$ &$\mathcal{B}_N$ \\
  \hline
  \multirow{2}*{quadratic}  &$\textbf{a}^{*}\approx(-0.1667, 1.3333)$&\multirow{2}*{$1.04^{\circ}$}\\
                            &$\textbf{b}^{*}\approx(-0.1078, -0.1347, 0.7273)$&\\
  \hline
  \multirow{2}*{cubic}      &$\textbf{a}^{*}\approx(0.0086, 1.3453, -0.4170, 0.0632)$&\multirow{2}*{$11.19^{\circ}$}\\
                            &$\textbf{b}^{*}\approx(0.0420, -0.1273, 0.6377)$&\\
  \hline
  \multirow{4}*{quartic}    &$\textbf{a}^{*}\approx(0.0829, 0.6149, 0.0970, 0.1238, 0.0815,$&\multirow{4}*{$28.85^{\circ}$}\\
                            &$0.0730, 0.0714, 0.7113)$        &\\
                            &$\textbf{b}^{*}\approx(-0.0276, 0.0169, -0.1193, 0.0087, -0.0268,$ &\\
                            &$-0.0008, -0.0712, 0.0428, 0.1493)$&\\
  \bottomrule
  \end{tabular}
\end{table}

\begin{theorem}[Coercivity] 
\label{thm:Coercivity}
If the triangular mesh $\mathcal{T}_h$ satisfies the minimum angle condition
\begin{align}
\label{eq:minmum_angleCondition}
\theta_K \geq \mathcal{B}_N, \quad \forall K \in \mathcal{T}_h,
\end{align}
there exists a positive constant $\alpha$ such that
\begin{align}
\label{eq:Coercivity equation}
a^{*}(u_h,\Pi_h^{*}u_h)\geq \alpha|u_h|_1^2, \quad \forall u_h\in U_h,
\end{align}
where $\theta_K$ denotes the minimal interior angle of $K$, the value of $\mathcal{B}_N$ is given in Table~\ref{tab:minimum-angle},
and $\Pi_h^{*}$ is the trial-to-test mapping with the optimal parameters $\textbf{a}^{*}$ and $\textbf{b}^{*}$.
\end{theorem}

\begin{proof}
Essentially, (\ref{eq:Coercivity equation}) is ensured by the minimum angle condition and
the definition of $\Pi_h^{*}$ and $\mathcal{B}_N$.
The proof of this result is similar to that of Theorem~4.14 in \cite{Chen.2012}. 
\end{proof}

\subsection{Boundedness}\label{subsec:Boundedness}

\begin{lemma}
\label{lem:boundedness lemma 1}
If the triangular mesh $\mathcal{T}_{h}$ satisfies the minimum angle condition (\ref{eq:minmum_angleCondition}), there holds
\begin{align}   \label{eq:u_diff_estimate}
|u_h(P_i)-u_h(P_j)|\leq C|u_h|_{1,K},\quad i,j=1,\cdots,N_{K}, \quad \forall K \in \mathcal{T}_h, \quad \forall u_h \in U_h .
\end{align}
Moreover, with the relation
(\ref{eq:restictions_pih}) for the parametric transformation matrix $M_k(\mathbf{a},\mathbf{b})$, one has
\begin{align}
& |[\![\Pi_{\mathrm{I}}^{*}u_h]\!]_{l^{*}}|    \leq  C|u_h|_{1,K},  \quad \forall u_{h}\in U_{h}, 
\label{eq:wh_difference1}
\\
& |\Pi_{\mathrm{II}}^{*}u_h|_K |   \leq  C|u_h|_{1,K}, \quad \forall u_{h}\in U_{h},
\label{eq:wh_difference2}
\end{align}
where $\Pi_{\mathrm{I}}^{*}$ and $\Pi_{\mathrm{II}}^{*}$ denote the part of $\Pi_h^{*}$ on the first and second layers of the dual mesh
and $[\![\Pi_{\mathrm{I}}^{*}u_h]\!]_{l^{*}}$ denotes the jump of $\Pi_{\mathrm{I}}^{*}u_h$ over the dual boundary $l^{*}$
corresponding to the first dual layer restricted within $K$.
\end{lemma}

\begin{proof}
Since $\mathcal{T}_{h}$ is a regular triangular mesh satisfying the minimum angle condition (\ref{eq:minmum_angleCondition}),
one has the following equivalence between the semi-$H^1$ norms on $K$ and $\hat{K}$ for $u_h\in U_h$.
\begin{align*}
C_{1} |u_{h}|_{1,K} \leq |\hat{u}_{h}|_{1,\hat{K}}\leq C_{2} |u_{h}|_{1,K}.
\end{align*}
Denote by $P_{K,i}$ ($i = 1, ..., N_K$) the image of $P_{i}$ ($i = 1, ..., N_K$) under the affine mapping from $\hat{K}$ to $K$.
Then, for any $i,j=1,\cdots,N_{K}$, 
\begin{align*}
|u_h(P_{K,i})-u_h(P_{K,j})| = |\hat{u}_h(P_i)-\hat{u}_h(P_j)|\leq C|\hat{u}_{h}|_{1,\hat{K}}\leq C |u_{h}|_{1,K},
\end{align*}
which gives the estimate (\ref{eq:u_diff_estimate}).
 
 Recall that for any $u_h \in U_h$, we have $u_h |_K = \sum_{i=1}^{N_K} u_{K,i} \hat{\phi}_i$, where $u_{K,i} = u_h(P_{K,i})$.
 Let $v_h = (v_{h,\mathrm{I}},v_{h,\mathrm{II}}) = \Pi^{*} u_h$.
From (\ref{eq:VKI}) and (\ref{eq:VKII}), we have
\[
v_{h,\mathrm{I}}|_{K}=\sum_{P_{K,j}\in\mathcal{N}_{1}\cup\mathcal{N}_{2}\cup\mathcal{N}_{3}} v_{K,j}\hat{\psi}_j,    
\quad
v_{h,\mathrm{II}}|_{K}=\sum_{P_{K,j}\in\mathcal{N}_{4}} v_{K,j}\hat{\psi}_j, 
\]
where $v_{K,j}:=v_h(P_{K,j})$ ($j=1,\dots,N_{K}$) are given by (\ref{eq:trial-to-test-mapping-def}). 
Consider the jump of $v_{h,\mathrm{I}}$ on $l^{*}=Q_{K,i_{1}}\cap Q_{K,i_{2}}$ ($i_1,i_2\in\{1,2,3\},$ $i_1<i_2$).
Notice that  $[\![\Pi_{\mathrm{I}}^{*}u_h]\!]_{l^{*}} = [\![ v_{h,\mathrm{I}} ]\!]_{l^{*}}$ can be represented
as a combination of $v_h(P_{K,i})-v_h(P_{K,j})$ $(P_{K,i}\in\mathcal{N}_{i_1},\,P_{K,j}\in\mathcal{N}_{i_2})$.
Recalling from (\ref{eq:restictions_pih}) that the row sums (combination coefficients) of $M_{k}(\V{a},\V{b})$ corresponding
to $\mathcal{T}_{h,\mathrm{I}}^{*}$ are 1, the jump $[\![\Pi_{\mathrm{I}}^{*}u_h]\!]_{l^{*}}$ can finally be expressed
as a combination of $u_{K,i}-u_{K,j}$ ($i,j=1,\dots,N_{K}$). Thus, (\ref{eq:wh_difference1}) can be derived
from (\ref{eq:u_diff_estimate}).

Moreover, (\ref{eq:restictions_pih}) indicates the row sums of $M_{k}(\V{a},\V{b})$ corresponding to the second dual layer are 0, which means $\Pi_{\mathrm{II}}^{*}u_h$ on $K$ can be expressed as a combination of $u_{K,i}-u_{K,j}$ ($i,j=1,\dots,N_{K}$). Then, we have (\ref{eq:wh_difference2}).
\end{proof}

\begin{theorem}[Boundedness]
\label{thm:boundedness}
For the trial-to-test mapping $\Pi_h^{*}$ given by (\ref{eq:trial-to-test-mapping-def}), there holds
\begin{align*}
a^{*}(u_h,\Pi_h^{*}v_h)\leq C|u_h|_1 |v_h|_1, \quad \forall u_h, \, v_h \in U_h,
\end{align*}
for FVE-2L schemes.
\end{theorem}
\begin{proof}
The bilinear form (\ref{eq:FVE-scheme}) of FVE-2L schemes can be rewritten as
\begin{align}
a^{*}(u_h,\Pi_h^{*}v_h)= a_{\mathrm{I}}^{*}(u_h,\Pi_{\mathrm{I}}^{*}v_h)+a_{\mathrm{II}}^{*}(u_h,\Pi_{\mathrm{II}}^{*}v_h) 
    = E_{11}+E_{12}+E_{21}+E_{22},\label{eq:boundedness proof 1}
\end{align}
where 
\begin{align*}
E_{11}=&\sum_{K\in\mathcal{T}_h}\sum_{K_{\mathrm{I}}^{*}\in\mathcal{T}_{\mathrm{I}}^{*}}
            \iint_{K_{\mathrm{I}}^{*}\cap K}(\mathbb{D}\nabla u_h)\cdot\nabla(\Pi_{\mathrm{I}}^{*}v_h)\,\ud x\ud y,\\
E_{12}=&-\sum_{K\in\mathcal{T}_h}\sum_{K_{\mathrm{I}}^{*}\in\mathcal{T}_{\mathrm{I}}^{*}}
            \int_{\partial K_{\mathrm{I}}^{*}\cap K}(\mathbb{D}\nabla u_h)\cdot\vec{n}\,    (\Pi_{\mathrm{I}}^{*}v_h)\,\ud s,\\
E_{21}=&\sum_{K\in\mathcal{T}_h}\iint_{K}(\mathbb{D}\nabla u_h)\cdot\nabla(\Pi_{\mathrm{II}}^{*}v_h)\,\ud x\ud y,\\
E_{22}=&-\sum_{K\in\mathcal{T}_h}\int_{\partial K}(\mathbb{D}\nabla u_h)\cdot\vec{n}\,  (\Pi_{\mathrm{II}}^{*}v_h)\,\ud s.
\end{align*}
In the above equation, we have used $K_{\mathrm{II}}^{*}=K$.
It is not difficult to see that $E_{11}$ and $E_{21}$ can be bounded by
\begin{align*}
| E_{11}| \leq C |u_h|_1 |v_h|_1,      
\qquad 
| E_{21}| \leq C |u_h|_1 |v_h|_1. 
\end{align*}
Moreover, using Lemma~\ref{lem:boundedness lemma 1} and the trace theorem, one gets
\begin{align}
| E_{12}|  & \leq C \sum_{K\in\mathcal{T}_{h}}\sum_{l^{*}} (\int_{l^{*}}(\mathbb{D}\nabla u_h\cdot\vec{n})^2\,\ud s)^{1/2}\,
                (\int_{l^{*}}  | [\![ \Pi_{\mathrm{I}}^{*}v_h ]\!]_{l^{*}} |^{2}  \ud s)^{1/2}\nonumber\\
        & \leq Ch^{\frac{1}{2}}\sum_{K\in\mathcal{T}_{h}}\sum_{l^{*}}    |u_h|_{1,l^{*}}\,|v_h|_{1,K}
        \leq C|u_h|_1 |v_h|_1 ,    
        \notag    
        \\
| E_{22}|  &  \leq C \sum_{K\in\mathcal{T}_{h}} (\int_{\partial K}(\mathbb{D}\nabla u_h\cdot\vec{n})^2\,\ud s)^{\frac{1}{2}}\,
                (\int_{\partial K}  | \Pi_{\mathrm{II}}^{*}v_h |^{2}  \ud s)^{1/2}     \nonumber\\
        & \leq Ch^{\frac{1}{2}}\sum_{K\in\mathcal{T}_{h}}|u_h|_{1,\partial K}\,|v_h|_{1,K} \le C|u_h|_1 |v_h|_1  . 
        \notag
\end{align}

The conclusion of Theorem~\ref{thm:boundedness} follows form the above results.
\end{proof}

\section{Error estimates}
\label{sec:L2 estimate}

Now we are ready to establish the $H^1$ and $L^2$ error estimates for $k$th-order ($k=2,\, 3, \, 4$) FVE-2L schemes.
It is worth emphasizing that the regularity requirement for the $L^2$ estimate of FVE-2L schemes is $u\in H^{k+1}$,
which is weaker than $u\in H^{k+2}$ required by the existing high-order FVE schemes
\cite{Li.2000,Lin.2015,Lv.2012,Wang.2016}.

\subsection{$H^1$ estimate}\label{sub:stability and H1 estimate}

\begin{theorem}[$H^1$ estimate]
\label{thm:H1 estimate}
Given a $k$th-order FVE-2L scheme ($k=2,\, 3, \, 4$) on a regular triangular mesh $\mathcal{T}_{h}$
of $\Omega$, let $u\in H_0^1(\Omega)\cap H^{k+1}$ and $u_h\in U_h$ be the solutions of (\ref{eq:second_order_elliptic_problem})
and (\ref{eq:FVE-scheme}), respectively. If $\mathcal{T}_h$ satisfies the minimum angle condition (\ref{eq:minmum_angleCondition}),
there holds
\begin{align}\label{eq:H1 estimate equation}
\|u-u_h\|_1\leq Ch^{k}\|u\|_{k+1},
\end{align}
where $C$ is a constant independent of $h$ and $u$.
\end{theorem}

\begin{proof}
The orthogonality for the bilinear form $a^{*}(\cdot,\cdot)$ of FVE-2L schemes reads as
\begin{align}\label{eq:orthogonality equation}
a^{*}(u-u_h,v_h)=0,\quad \forall v_h\in V_h.
\end{align}
Taking $v_h= \Pi_h^{*}(u_h-\Pi_h^{k}u)$ in the above equation, we get
\begin{align*}
a^{*}(u-u_h,\Pi_h^{*}(u_h-\Pi_h^{k}u))=0 .
\end{align*}
Using this and Theorems~\ref{thm:Coercivity} and \ref{thm:boundedness}, we have
\begin{align}\label{eq:H1 eatimate proof 1}
|u_h-\Pi_h^{k}u|_1&\leq \frac{1}{\alpha}\frac{a^{*}(u_h-\Pi_h^ku,\Pi_h^{*}(u_h-\Pi_h^{k}u))}{|u_h-\Pi_h^{k}u|_1} \nonumber\\
&=\frac{1}{\alpha}\frac{a^{*}(u-\Pi_h^ku,\Pi_h^{*}(u_h-\Pi_h^{k}u))}{|u_h-\Pi_h^{k}u|_1} \nonumber\\
&\leq C\frac{|u-\Pi_h^ku|_1\,|u_h-\Pi_h^{k}u|_1}{|u_h-\Pi_h^{k}u|_1}\nonumber\\
&\leq C |u-\Pi_h^ku|_1.
\end{align}
Then,
\begin{align*}
|u-u_h|_1\leq|u-\Pi_h^ku|_1+|u_h-\Pi_h^ku|_1\leq C|u-\Pi_h^ku|_1\leq Ch^k\|u\|_{k+1},
\end{align*}
which gives (\ref{eq:H1 estimate equation}).
\end{proof}

\subsection{$L^2$ estimate}
\label{subsec:L2}

Consider the auxiliary problem: for any $g\in L^2(\Omega)$, find $\omega\in H_0^1(\Omega)$ such that
\begin{align}\label{eq:auxiliary equation}
a(v,\omega)=(g,v),\quad \forall v\in H_0^1(\Omega),
\end{align}
where $a(\cdot,\cdot)$ is the standard bilinear form of the finite element method. 
This problem has a unique solution $\omega\in H_0^1(\Omega)\cap H^2(\Omega)$ satisfying
\begin{align}
\label{eq:regular equation}
||\omega||_2\leq C||g||_0.
\end{align}

Recall that the first dual layer $\mathcal{T}_{\mathrm{I}}$ constitutes a complete partition of $\Omega$.
Taking $v_{h,II}\equiv0$ ($v_{h}=(v_{I},v_{II})$) in (\ref{eq:orthogonality equation}), one also has the orthogonality
on the first dual layer, i.e.,
\begin{align}
\label{eq:auxiliary first layer equation}
a_{\mathrm{I}}^{*}(u-u_h,\Pi_{\mathrm{I}}^{*}  (\Pi_{h}^{1}\omega)  )=0\quad \forall \omega\in H_0^1(\Omega).
\end{align}

\begin{theorem}($L^2$ estimate)
\label{thm:L2 estimate}
Under the same assumptions as in Theorem~\ref{thm:H1 estimate},
the following estimate holds for each $k$th-order FVE-2L scheme ($k=2, \, 3, \, 4$),
\begin{align}
\label{eq:L2 estimate equation}
||u-u_h||_0\leq C h^{k+1}\|u\|_{k+1}.
\end{align}
\end{theorem}

\begin{proof}
Letting $g=v=u-u_h$ in (\ref{eq:auxiliary equation}) and combining it with (\ref{eq:auxiliary first layer equation}), there holds
\begin{align}\label{eq:L2 estimate proof 1}
||u-u_h||_0^2&=a(u-u_h,\omega)\nonumber\\
&=a(u-u_h,\omega-\Pi_h^1\omega)+a(u-u_h,\Pi_h^1\omega)-a_{\mathrm{I}}^{*}(u-u_h,\Pi_{\mathrm{I}}^{*}  (\Pi_{h}^{1}\omega)  )\nonumber\\
&=I_1+I_2,
\end{align}
where
\begin{align*}
I_1&=a(u-u_h,\omega-\Pi_h^1\omega),\\
I_2&=a(u-u_h,\Pi_h^1\omega)-a_{\mathrm{I}}^{*}(u-u_h,\Pi_{\mathrm{I}}^{*}  (\Pi_{h}^{1}\omega)  ).
\end{align*}

For $I_1$, using Theorem~\ref{thm:H1 estimate} we get
\begin{align}\label{eq:L2 estimate proof 2}
|I_1|&=|\iint_{\Omega}(\mathbb{D}\nabla(u-u_h))\cdot\nabla(\omega-\Pi_h^1\omega)\,\ud x\ud y|\nonumber\\
&\leq C|u-u_h|_1\,|\omega-\Pi_h^1\omega|_1\nonumber\\
&\leq C h^{k+1}\|u\|_{k+1}\,\|\omega\|_2.
\end{align}

For $I_2$, using the divergence theorem, one has
\begin{align}
a(u-u_h,\Pi_h^1\omega)  =&\iint_{\Omega}(\mathbb{D}\nabla(u-u_h))\cdot\nabla(\Pi_h^1\omega)\,\ud x\ud y\nonumber\\
        =&\sum_{K\in\mathcal{T}_h}  \Big(-\iint_{K}\nabla\cdot(\mathbb{D}\nabla(u-u_h))(\Pi_h^{1}\omega)\,\ud x\ud y\nonumber\\
            &\qquad+\int_{\partial K}(\mathbb{D}\nabla(u-u_h))\cdot\vec{n}\,(\Pi_h^1\omega)\,\ud s  \Big),      \label{eq:L2 estimate proof 3}\\
a_{\mathrm{I}}^{*}(u-u_h,\Pi_{\mathrm{I}}^{*}\omega)
    =&\sum_{K\in\mathcal{T}_h}\sum_{K_{\mathrm{I}}^{*}\in\mathcal{T}_{\mathrm{I}}^{*}}
        \Big(\iint_{K_{\mathrm{I}}^{*}\cap K}(\mathbb{D}\nabla(u-u_h))\cdot\nabla(\Pi_{\mathrm{I}}^{*}  (\Pi_{h}^{1}\omega)  )\,\ud x\ud y\nonumber\\
     &\qquad    -\int_{\partial K_{\mathrm{I}}^{*}\cap K}(\mathbb{D}\nabla(u-u_h))\cdot\vec{n}\,(\Pi_{\mathrm{I}}^{*}  (\Pi_{h}^{1}\omega) )\,\ud s     \Big)\nonumber\\
    =&\sum_{K\in\mathcal{T}_h}\Big(-\iint_{K}\nabla\cdot(\mathbb{D}\nabla(u-u_h))(\Pi_{\mathrm{I}}^{*}   (\Pi_{h}^{1}\omega)  )\,\ud x\ud y\nonumber\\
     & \qquad   +\int_{\partial K}(\mathbb{D}\nabla(u-u_h))\cdot\vec{n}\,(\Pi_{\mathrm{I}}^{*}   (\Pi_{h}^{1}\omega)   )\,\ud s    \Big).      \label{eq:L2 estimate proof 4}
\end{align}
Notice that the test space $V_{K,I}$ contains the piecewise linear space over $K=Q_{K,1}\cup Q_{K,2}\cup Q_{K,3}$. The mapping $\Pi_{\mathrm{I}}^{*}$ maps the linear function $\Pi_{h}^{1}\omega$ into itself on $Q_{K,i}$ ($i=1,2,3$), which implies $\Pi_{\mathrm{I}}^{*}   (\Pi_{h}^{1}\omega) = \Pi_{h}^{1}\omega$ on all $K\in\mathcal{T}_{h}$. 
Combining (\ref{eq:L2 estimate proof 3}) with (\ref{eq:L2 estimate proof 4}), one has
\begin{align}
I_{2} = 0.
\end{align}
The above estimate together with (\ref{eq:auxiliary equation}), (\ref{eq:L2 estimate proof 1}), and (\ref{eq:L2 estimate proof 2}) leads to
\begin{align}\label{eq:L2 estimate proof 7}
||u-u_h||_0^2\leq |I_1|+|I_2|\leq Ch^{k+1}\|u\|_{k+1}\,\|\omega\|_2\leq Ch^{k+1}\|u\|_{k+1}\,||u-u_h||_0,
\end{align}
which yields (\ref{eq:L2 estimate equation}).
\end{proof}

\section{Numerical experiments}
\label{sec:numerical experiments}

In this section we present numerical results to demonstrate the performance of FVE-2L schemes ($k$=2,\, 3,\, 4)
on triangular meshes for elliptic and linear elasticity problems. We focus on  the conservation properties on the two layers of the dual mesh
and the $H^1$ and $L^2$ convergence rates of FVE-2L schemes. We also compare the condition number between
the FVE-2L schemes, single layer FVE schemes (FVEM) from \cite{Wang.2016}, and finite element schemes (FEM)
for Example~\ref{ex:Elliptic problem}. The triangular mesh is obtained from a rectangular mesh by partitioning each
rectangular element into two triangular elements.

\begin{example}[Elliptic problem]
\label{ex:Elliptic problem}
Consider the elliptic problem
\begin{align}\label{eq:Elliptic_equation}
\left\{
\begin{array}{cl}
-\nabla\cdot(\mathbb{D}\nabla u)= -5\exp(x+2y),& \mathrm{in}\,\,\Omega = (-1,1)\times (-1, 1),    \\
u=\exp(x+2y), &  \mathrm{on}\,\,\partial\Omega.
\end{array}
\right.
\end{align}
Here, $\mathbb{D}=\mathbb{I}$ and the exact solution is $u=\exp(x+2y)$.
Define the local conservation errors on each dual element $K^{*}\in\{K_{I}^{*},\,K_{II}^{*}\}$ in flux and equation forms as
\begin{align*}
C_{K^{*},flux} & =-\int_{\partial K^{*}}(\mathbb{D}\nabla u_h)\cdot\vec{n}\,\ud s-\iint_{K^{*}}f\,\ud x\ud y,  \\
C_{K^{*},equa} & =-\iint_{K^{*}}\nabla\cdot(\mathbb{D}\nabla u_h)\,\ud x \ud y-\iint_{K^{*}}f\,\ud x\ud y.
\end{align*}
The corresponding global conservation errors on $\mathcal{T}^{*}\in\{\mathcal{T}_{I}^{*},\,\mathcal{T}_{II}^{*}\}$ are given by
\begin{align*}
C_{\mathcal{T}^{*},flux} & =  \sum_{K^{*}\in\mathcal{T}^{*}}C_{K^{*},flux},  \\
C_{\mathcal{T}^{*},equa} & = \sum_{K^{*}\in\mathcal{T}^{*}}C_{K^{*},equa}.
\end{align*}

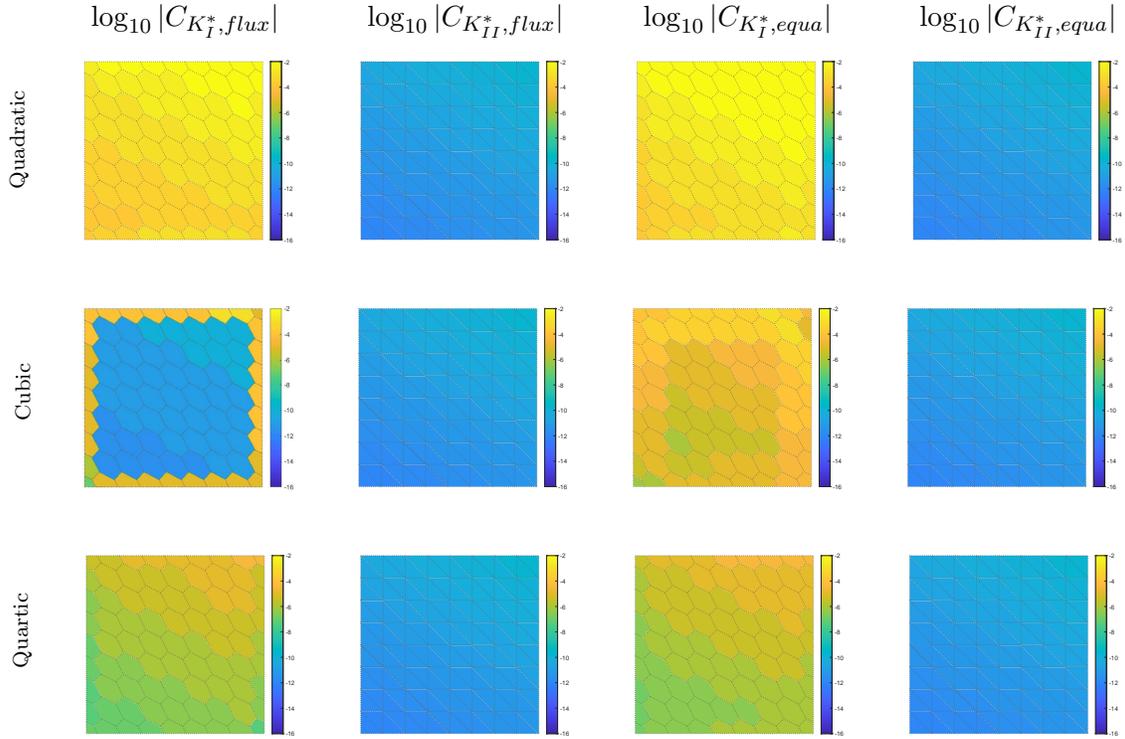
\begin{figure}[!htbp]
    \centering
    \subfigure{
    \rotatebox{90}{\scriptsize{~~~~~~~~~~Quadratic}}
    \begin{picture}(-3,-3)
    \put(20,90){$\log_{10} |C_{K_{I}^{*},flux}|$}
    \end{picture}
    \begin{minipage}[t]{.20\textwidth}
      \centering
      \includegraphics*[width=110pt]{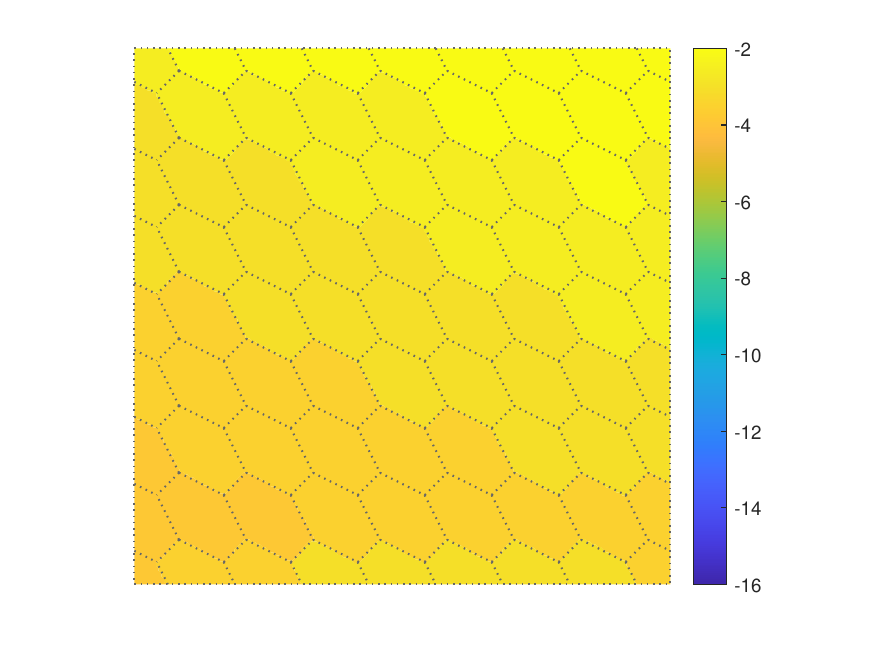}
    \end{minipage}}
    \subfigure{
    \begin{picture}(-3,-3)
    \put(20,90){$\log_{10} |C_{K_{II}^{*},flux}|$}
    \end{picture}
    \begin{minipage}[t]{.20\textwidth}
    \centering
      \includegraphics[width=110pt]{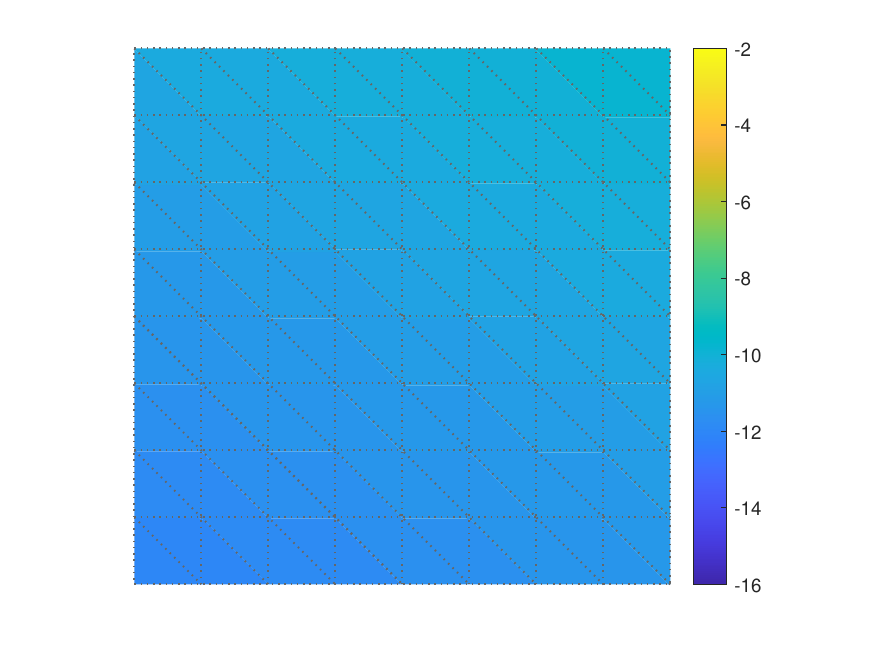}
    \end{minipage}}
    \subfigure{
    \begin{picture}(-3,-3)
    \put(20,90){$\log_{10} |C_{K_{I}^{*},equa}|$}
    \end{picture}
    \begin{minipage}[t]{.20\textwidth}
    \centering
      \includegraphics[width=110pt]{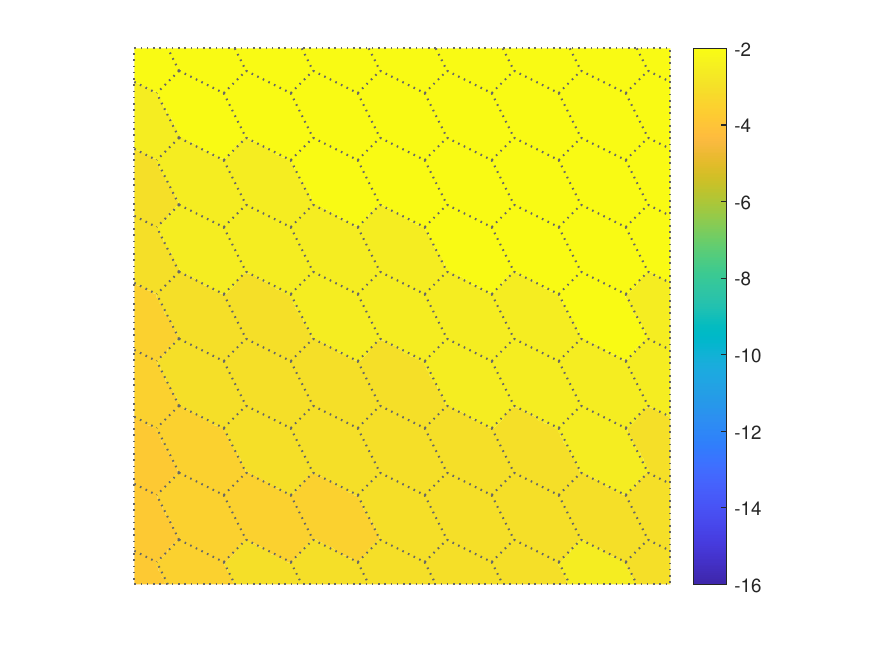}
    \end{minipage}}
    \subfigure{
    \begin{picture}(-3,-3)
    \put(20,90){$\log_{10} |C_{K_{II}^{*},equa}|$}
    \end{picture}
    \begin{minipage}[t]{.20\textwidth}
    \centering
      \includegraphics[width=110pt]{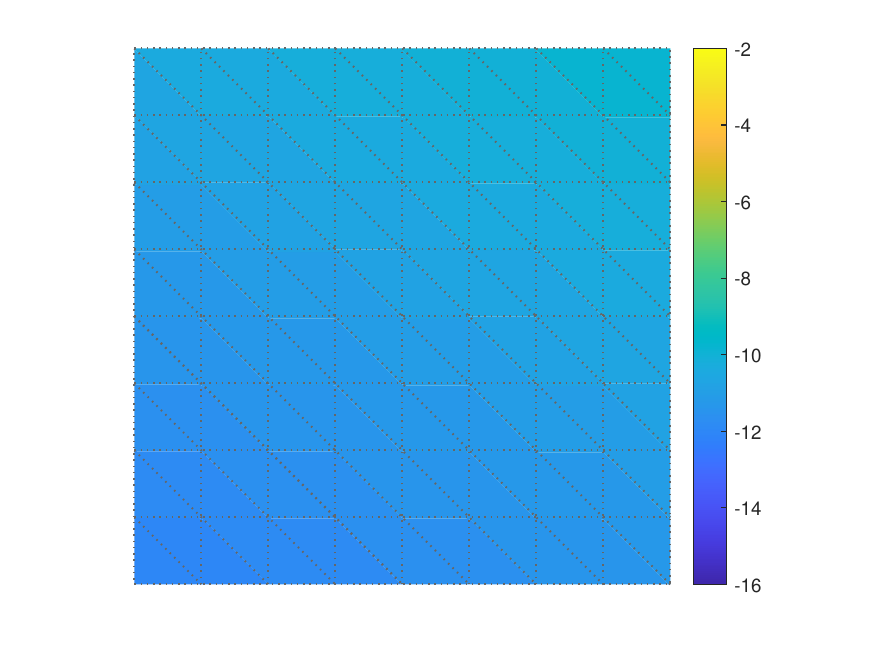}
    \end{minipage}}
    \qquad
    \subfigure{
    \rotatebox{90}{\scriptsize{~~~~~~~~~~~~Cubic}}
    \begin{minipage}[t]{.20\textwidth}
      \centering
      \includegraphics*[width=110pt]{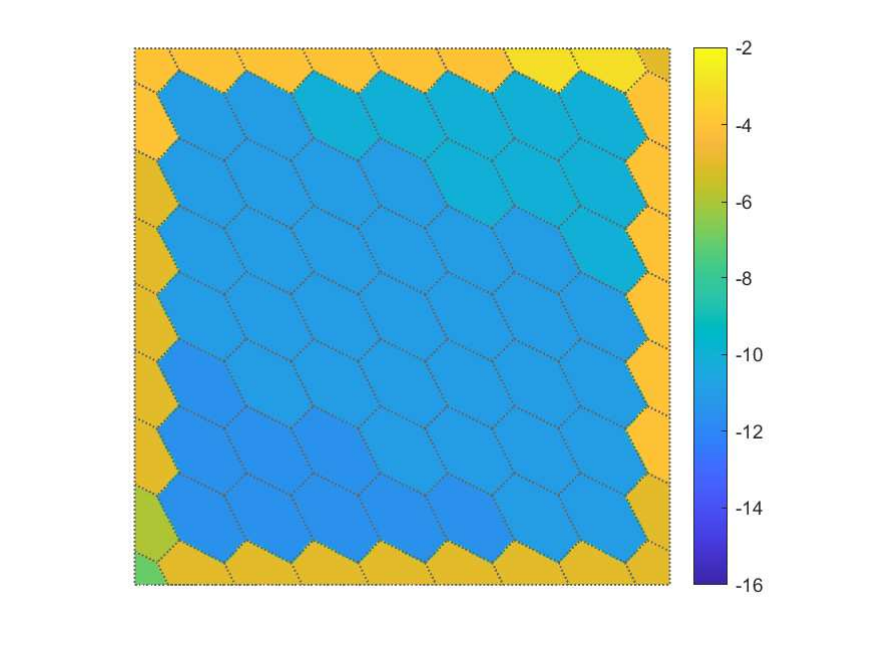}
    \end{minipage}}
    \subfigure{
    \begin{minipage}[t]{.20\textwidth}
    \centering
      \includegraphics[width=110pt]{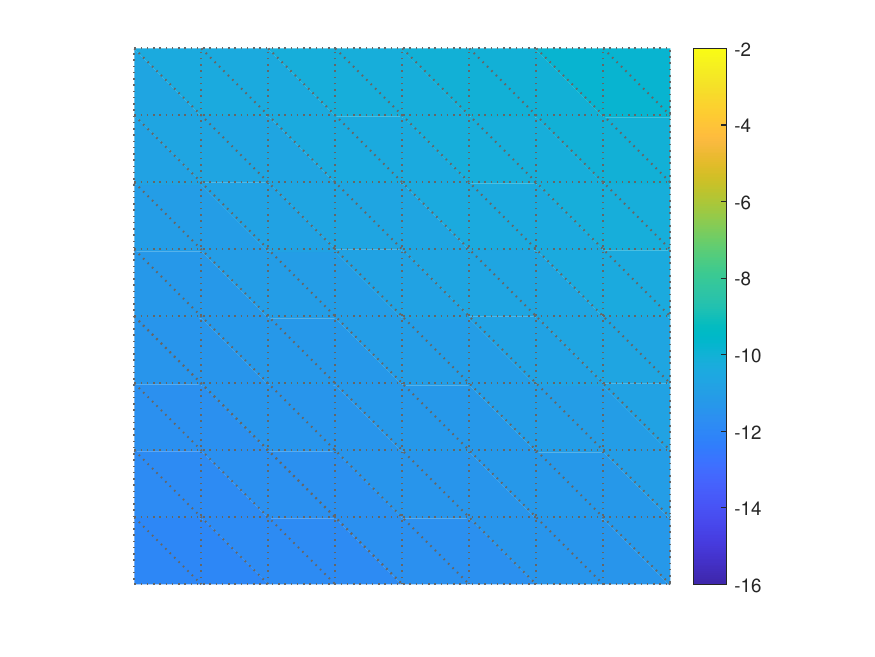}
    \end{minipage}}
    \subfigure{
    \begin{minipage}[t]{.20\textwidth}
    \centering
      \includegraphics[width=110pt]{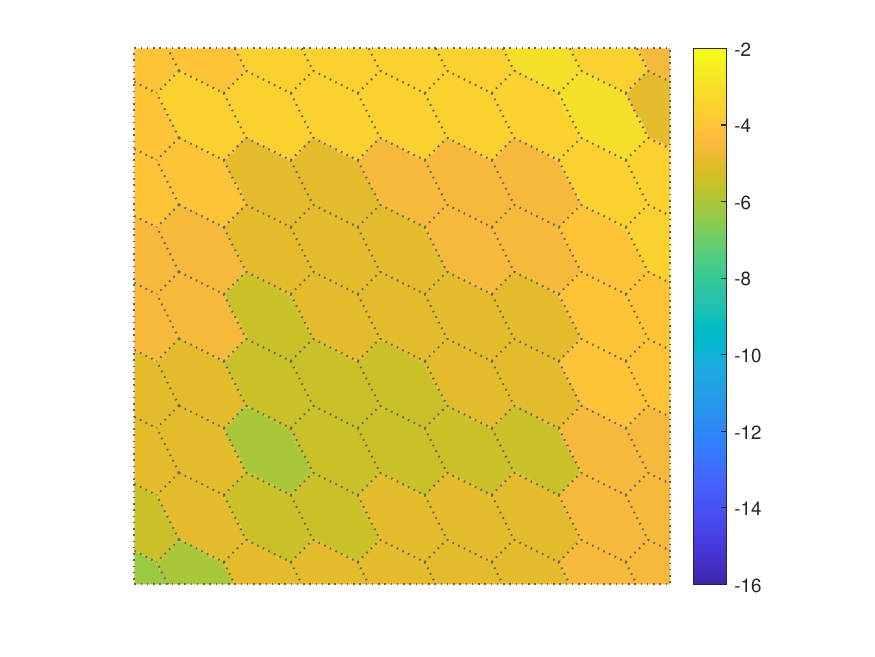}
    \end{minipage}}
    \subfigure{
    \begin{minipage}[t]{.20\textwidth}
    \centering
      \includegraphics[width=110pt]{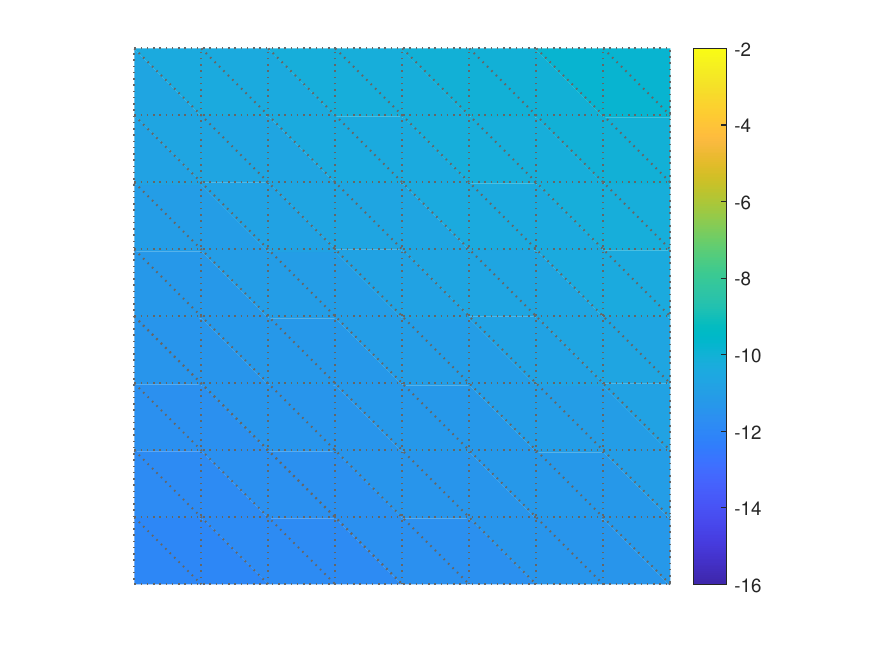}
    \end{minipage}}
    \qquad
    \subfigure{
    \rotatebox{90}{\scriptsize{~~~~~~~~~~~~Quartic}}
    \begin{minipage}[t]{.20\textwidth}
      \centering
      \includegraphics*[width=110pt]{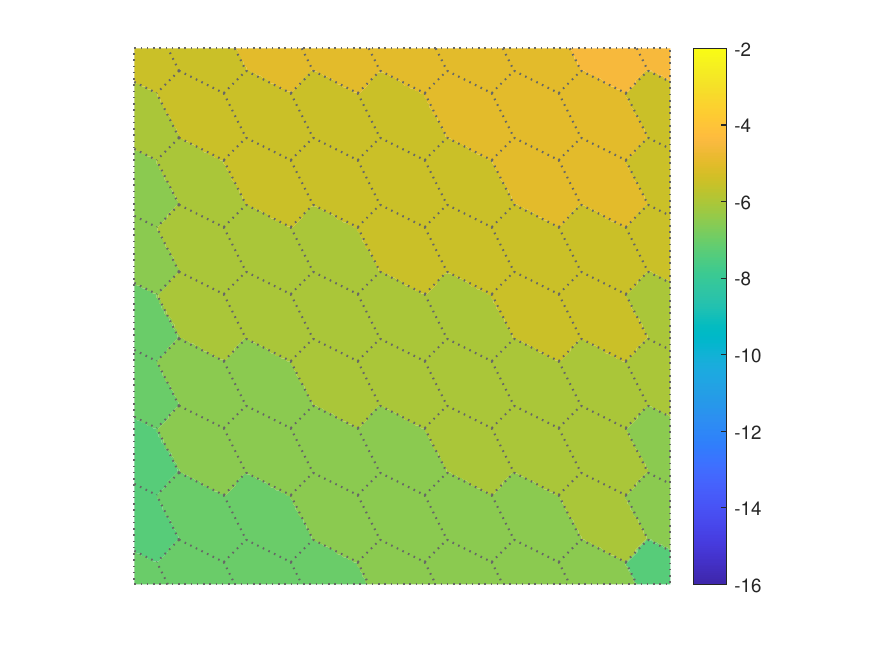}
    \end{minipage}}
    \subfigure{
    \begin{minipage}[t]{.20\textwidth}
    \centering
      \includegraphics[width=110pt]{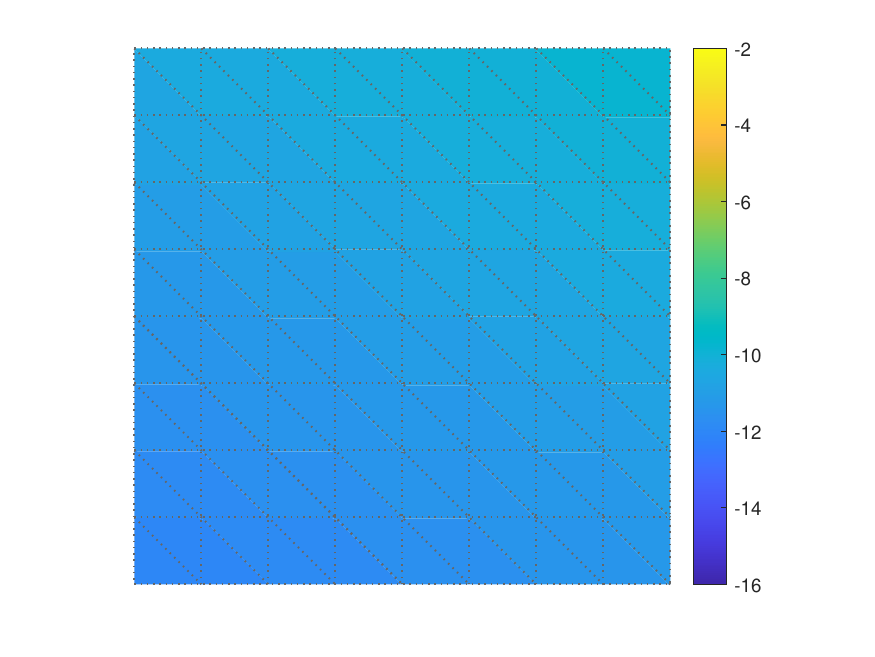}
    \end{minipage}}
    \subfigure{
    \begin{minipage}[t]{.20\textwidth}
    \centering
      \includegraphics[width=110pt]{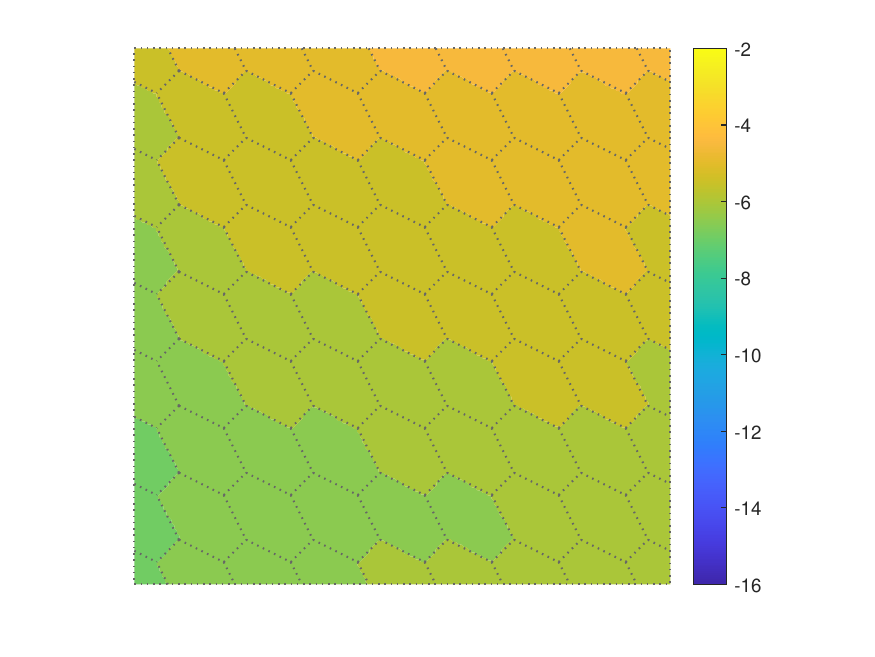}
    \end{minipage}}
    \subfigure{
    \begin{minipage}[t]{.20\textwidth}
    \centering
      \includegraphics[width=110pt]{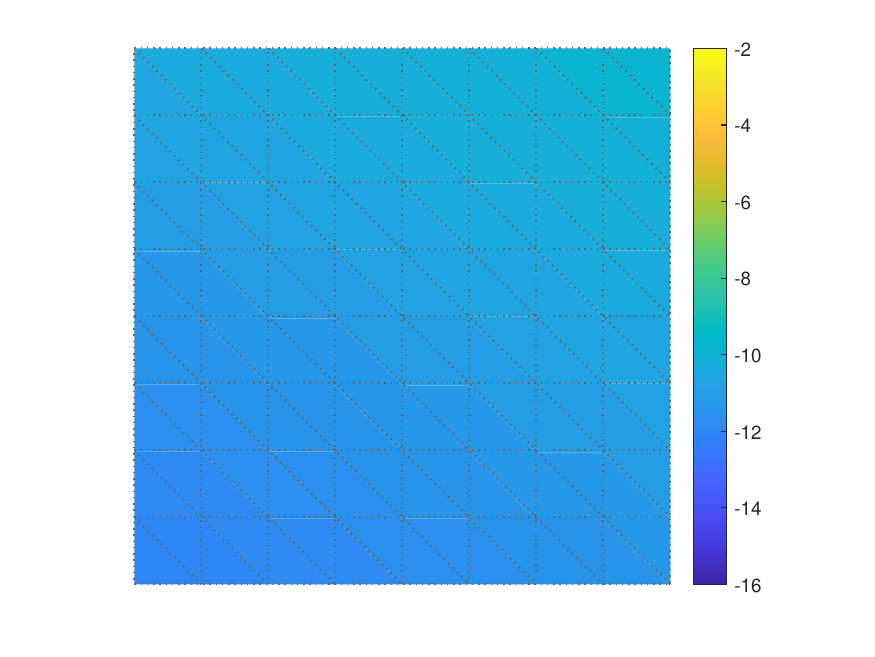}
    \end{minipage}}
    \caption{The conservation properties of the FVE-2L schemes in Example~\ref{ex:Elliptic problem}. The dual elements marked blue indicate that local conservation is maintained, while the dual elements marked yellow indicate the violation of the local conservation law. A mesh with
   $h \approx 1/8$ is used.}
    \label{fig:Conservation_error_ex1}
\end{figure}

\begin{figure}[!htbp]
    \centering
    \subfigure{
    \begin{minipage}[t]{.47\textwidth}
      \centering
      {(a) The $H^1$ convergence rate}\\
      \includegraphics*[width=180pt]{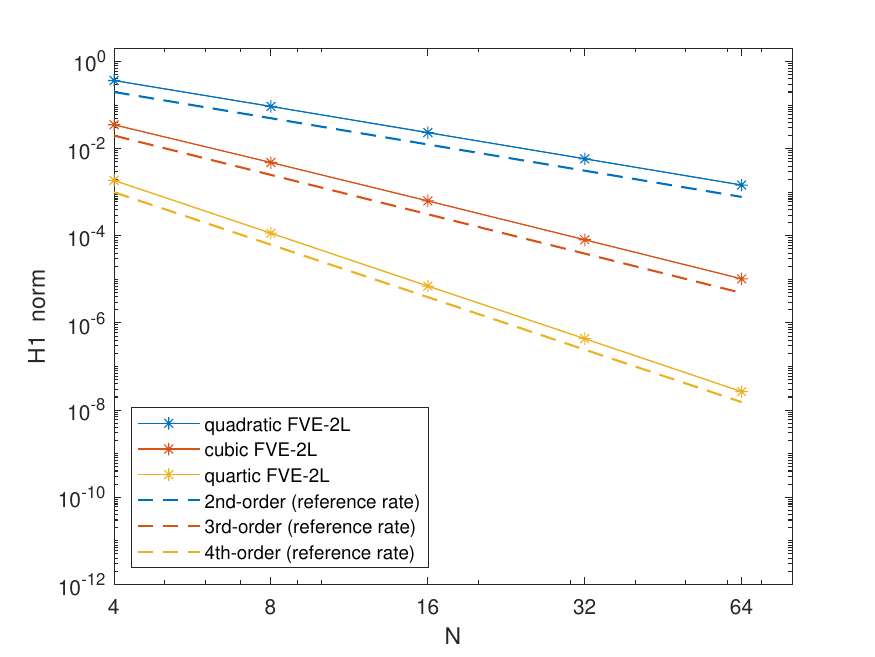}
    \end{minipage}}
    \subfigure{
    \begin{minipage}[t]{.47\textwidth}
    \centering
    {(b) The $L^2$ convergence rate}\\
      \includegraphics[width=180pt]{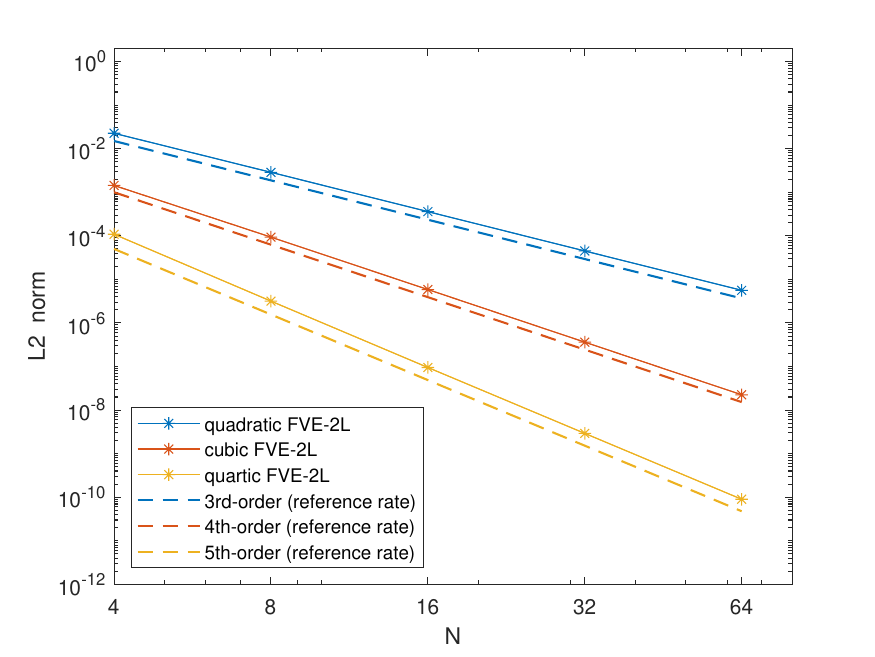}
    \end{minipage}}
    \caption{The numerical convergence rates of FVE-2L schemes for Example~\ref{ex:Elliptic problem}. Here, $N$ denotes
    the number of intervals in each axis direction of the primary mesh.}
    \label{fig:convergence_ex1}
\end{figure}

As shown in Fig.~\ref{fig:Conservation_error_ex1}, on the second dual layer $\mathcal{T}_{II}^{*}$, the FVE-2L schemes preserve the local conservation law in both flux and equation forms (the second and forth columns). On the other hand, the FVE-2L schemes violate the local conservation law in equation form on the first dual layer $\mathcal{T}_{I}^{*}$ (the third column).
Moreover, the local conservation law in flux form on the first dual layer $\mathcal{T}_{I}^{*}$ is maintained for odd-order schemes (on interior dual elements)
while being violated by even-order schemes (the first column).

From Table~\ref{tab:conservation_elliptic}, one can see that
the global conservation law is preserved in both flux form and equation form on $\mathcal{T}_{II}^{*}$
and only in equation form on $\mathcal{T}_{I}^{*}$.
Recall that the existing FVE schemes use a single layer for the dual mesh. They only maintain the local conservation law in flux form on interior dual elements (similar to the first figure in the second row of Fig.~\ref{fig:Conservation_error_ex1}) and do not preserve the local conservation law in equation form nor the global conservation law in either flux or equation form.

In Fig.~\ref{fig:convergence_ex1}, the optimal $H^1$ and $L^2$ convergence rates of the FVE-2L schemes are verified.

In Fig.~\ref{fig:E1_conditionNumber} the condition number of the stiffness matrix for the FVE-2L, FVEM, and FEM schemes
is plotted as a function of $N$, the number of intervals in each axial direction of the primary mesh.
Here, the condition number of a stiffness matrix $A$ is defined as (e.g., see \cite{Wang.2019})
\begin{equation}
{\kappa}(A) = \frac{\sigma_{max}(A)}{\lambda_{\min}((A + A^T)/2)} ,
\label{cond-0}
\end{equation}
where $\sigma_{max}(A)$ is the largest singular value and $\lambda_{\min}((A + A^T)/2)$
is the minimal eigenvalue of the symmetric part of $A$. It is known \cite{Eisenstat.1983} that the asymptotic convergence factor of
the generalized minimal residual method (GMRES) applied to linear systems associated with $A$ is bounded
by $\sqrt{1-\kappa(A)^{-2}}$. 
Fig.~\ref{fig:E1_conditionNumber} shows that the condition number of the stiffness matrix for the FVE-2L schemes
is slightly larger than those of FVEM and FEM but otherwise has the same growth rate $\mathcal{O}(\bar{h}^{-2})$,
where $\bar{h} = 1/\sqrt{N}$ is the average size of mesh elements. 

\begin{table}[htbp!]
\centering
  \caption{The global conservation law in Example~\ref{ex:Elliptic problem}}\label{tab:global conservation law for ex1}
  \label{tab:conservation_elliptic}
  \begin{tabular}{ c | c  c c c  }
  \toprule
  FVE-2L scheme & $C_{\mathcal{T}_{I,flux}^{*}}$ & $C_{\mathcal{T}_{II,flux}^{*}}$ & $C_{\mathcal{T}_{I,equa}^{*}}$ & $C_{\mathcal{T}_{II,equa}^{*}}$  \\
  \hline
  quadratic & 0.3374        & \textbf{4.2614e-09} & 4.2814e-09 & \textbf{4.2623e-09}\\
  cubic     & 0.0020        & \textbf{4.2613e-09} & 4.2816e-09 & \textbf{4.2627e-09}\\
  quartic   & 1.5682e-05   & \textbf{4.2633e-09} & 4.2805e-09 & \textbf{4.2612e-09}\\
  \bottomrule
  \end{tabular}
\end{table}

\begin{figure}[!htbp]
    \centering
    \subfigure{
    \begin{minipage}[t]{.32\textwidth}
      \centering
      {(a) Quadratic}\\
      \includegraphics*[width=120pt]{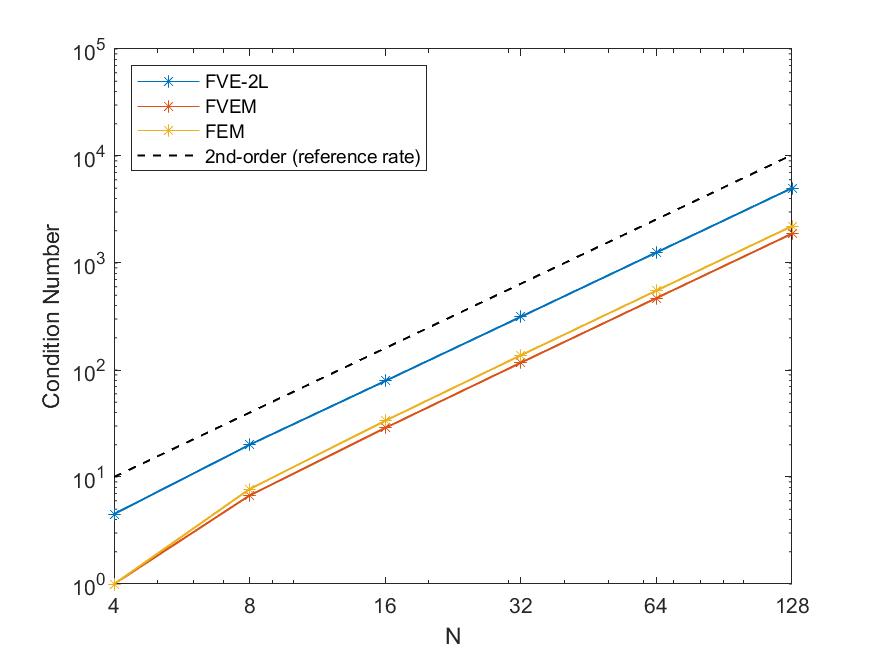}
    \end{minipage}}
    \subfigure{
    \begin{minipage}[t]{.32\textwidth}
    \centering
    {(b) Cubic}\\
      \includegraphics[width=120pt]{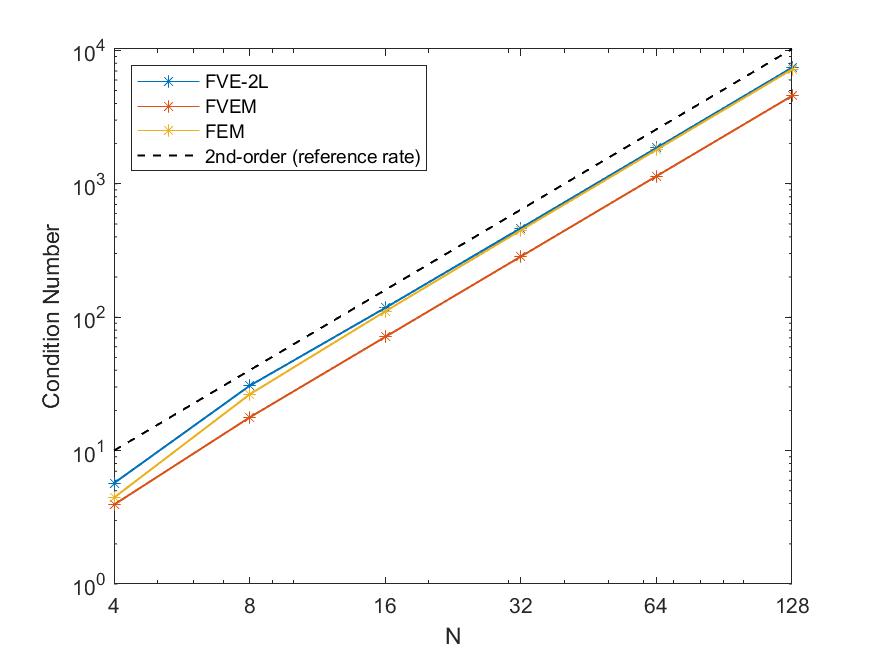}
    \end{minipage}}
   \subfigure{   
   \begin{minipage}[t]{.32\textwidth}
    \centering
    {(c) Quartic}\\
      \includegraphics[width=120pt]{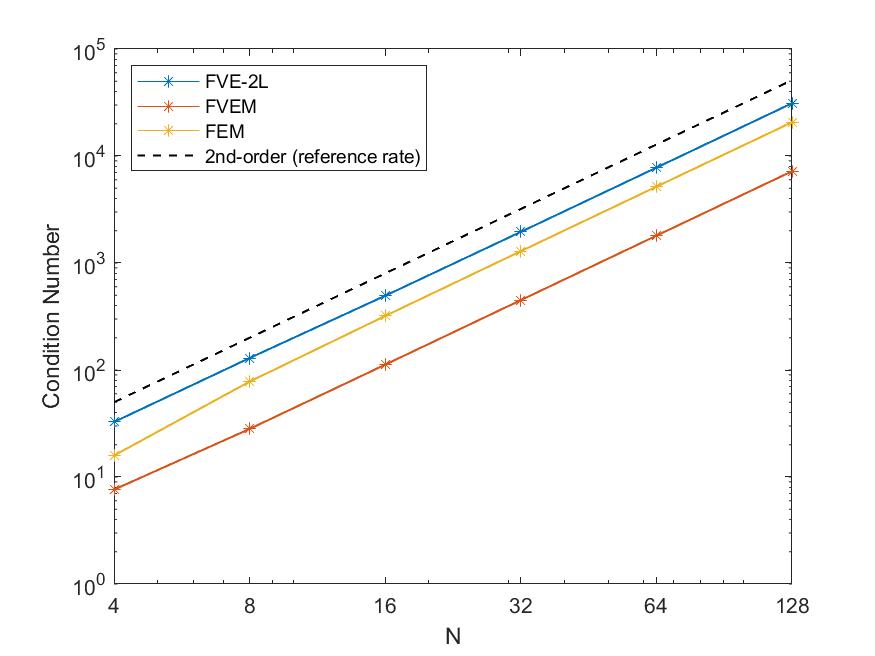}
    \end{minipage}}
    \caption{The condition number of the stiffness matrix for the FVE-2L, FVEM, and FEM schemes for Example~\ref{ex:Elliptic problem}
    is plotted as a function of $N$, the number of intervals in each axis direction of the primary mesh.}
    \label{fig:E1_conditionNumber}
\end{figure}

\end{example}

\begin{example}[Linear elasticity problem]\label{ex:Elasticity problem}
Consider the linear elasticity problem
\begin{align} \label{eq:Elasticity_equation}
\left\{
\begin{array}{cl}
-\nabla \cdot \sigma(\textbf{u})=\textbf{f}, & \mathrm{in}\,\,\Omega = (0,1)\times (0,1),    \\
\textbf{u}=(u_1,u_2)^{T}=(0,0)^{T}, &  \mathrm{on}\,\,\partial\Omega,
\end{array}
\right.
\end{align}
where $\textbf{f}$ is a given function,  the stress tensor $\sigma(\textbf{u})$ and the strain tensor $\epsilon(\textbf{u})$ are given by
\begin{align*}
\sigma(\textbf{u})=2\mu\epsilon(\textbf{u})+\lambda(\nabla\cdot\textbf{u})\mathbb{I},\qquad \epsilon(\textbf{u})=\frac{1}{2}(\nabla\textbf{u}+(\nabla\textbf{u})^{T}),
\end{align*}
and $\lambda$ and $\mu$ are Lam\'{e} constants.

\vspace{5pt}

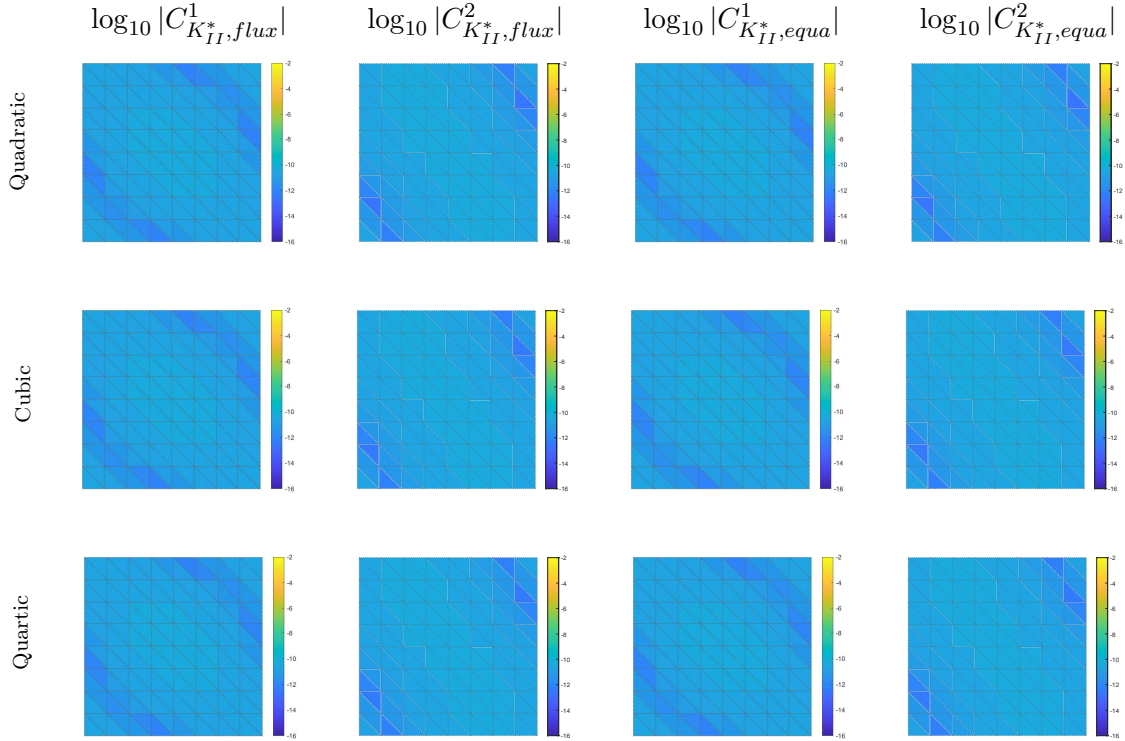
\begin{figure}[!htbp]
    \centering
    \subfigure{
    \rotatebox{90}{\scriptsize{~~~~~~~~~~Quadratic}}
    \begin{picture}(-3,-3)
    \put(20,90){$\log_{10} |C_{K_{II}^{*},flux}^{1}|$}
    \end{picture}
    \begin{minipage}[t]{.20\textwidth}
      \centering
      \includegraphics*[width=110pt]{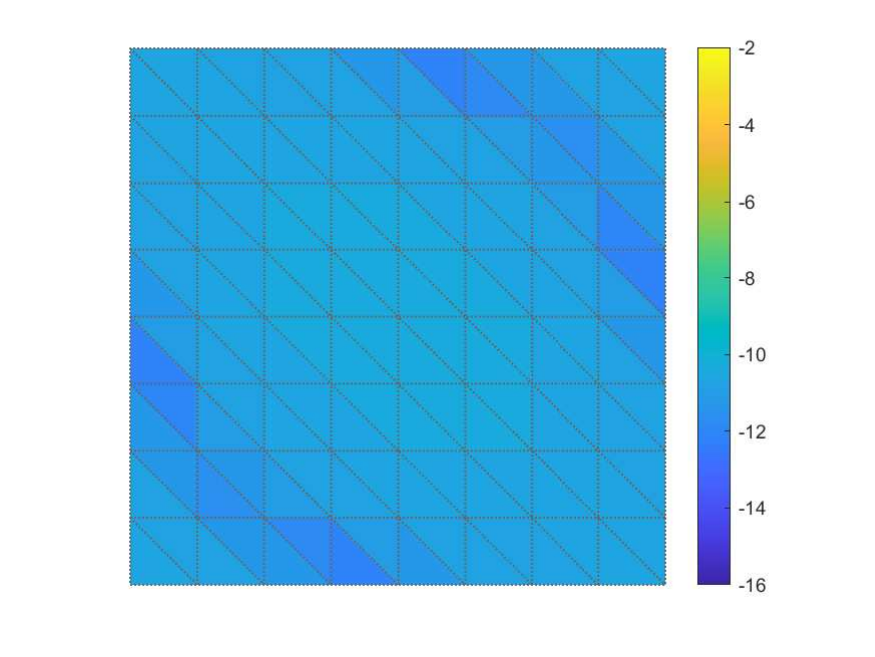}
    \end{minipage}}
    \subfigure{
    \begin{picture}(-3,-3)
    \put(20,90){$\log_{10} |C_{K_{II}^{*},flux}^{2}|$}
    \end{picture}
    \begin{minipage}[t]{.20\textwidth}
    \centering
      \includegraphics[width=110pt]{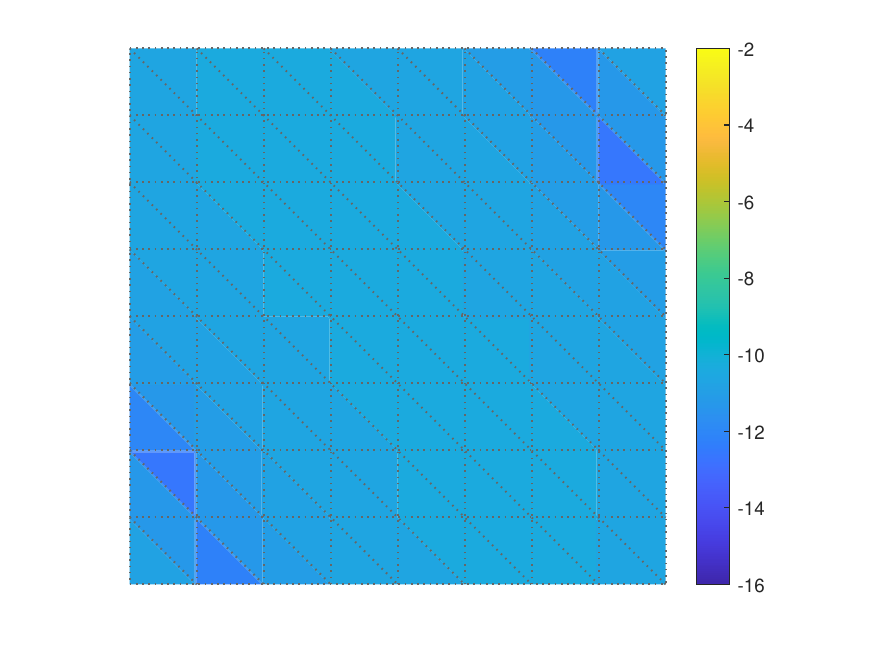}
    \end{minipage}}
    \subfigure{
    \begin{picture}(-3,-3)
    \put(20,90){$\log_{10} |C_{K_{II}^{*},equa}^{1}|$}
    \end{picture}
    \begin{minipage}[t]{.20\textwidth}
    \centering
      \includegraphics[width=110pt]{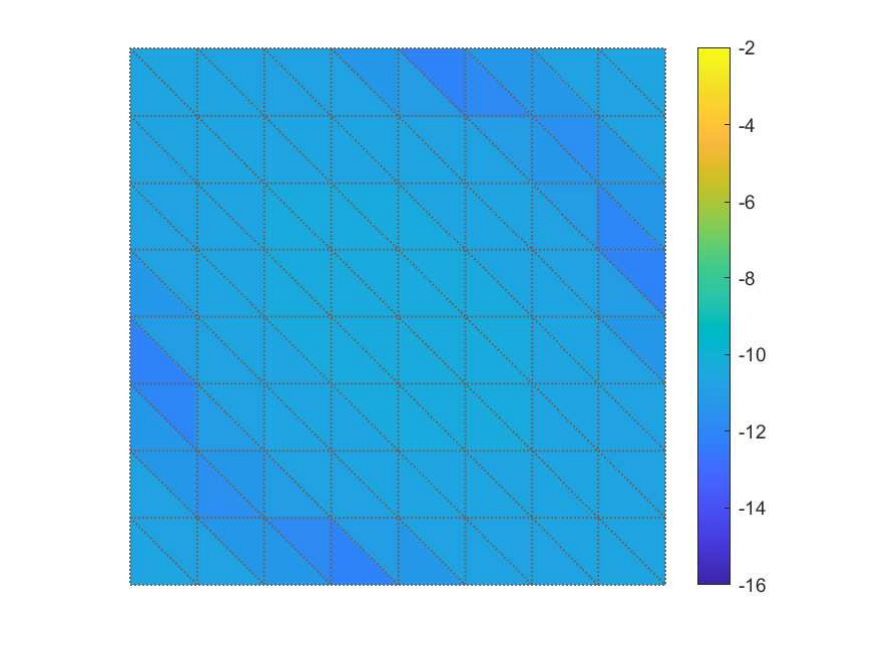}
    \end{minipage}}
    \subfigure{
    \begin{picture}(-3,-3)
    \put(20,90){$\log_{10} |C_{K_{II}^{*},equa}^{2}|$}
    \end{picture}
    \begin{minipage}[t]{.20\textwidth}
    \centering
      \includegraphics[width=110pt]{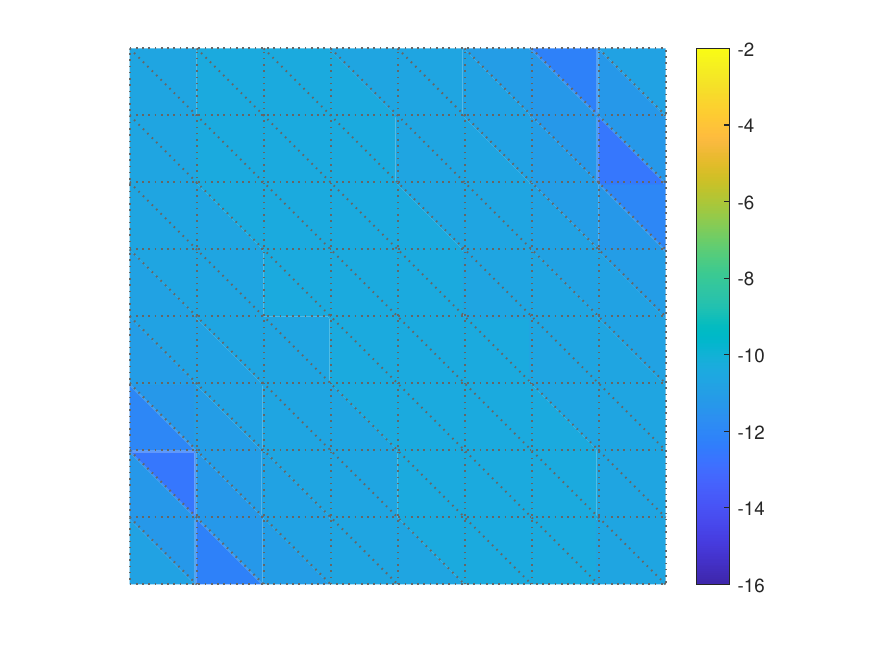}
    \end{minipage}}
    \qquad
    \subfigure{
   \rotatebox{90}{\scriptsize{~~~~~~~~~~~~Cubic}}
    \begin{minipage}[t]{.20\textwidth}
      \centering
      \includegraphics*[width=110pt]{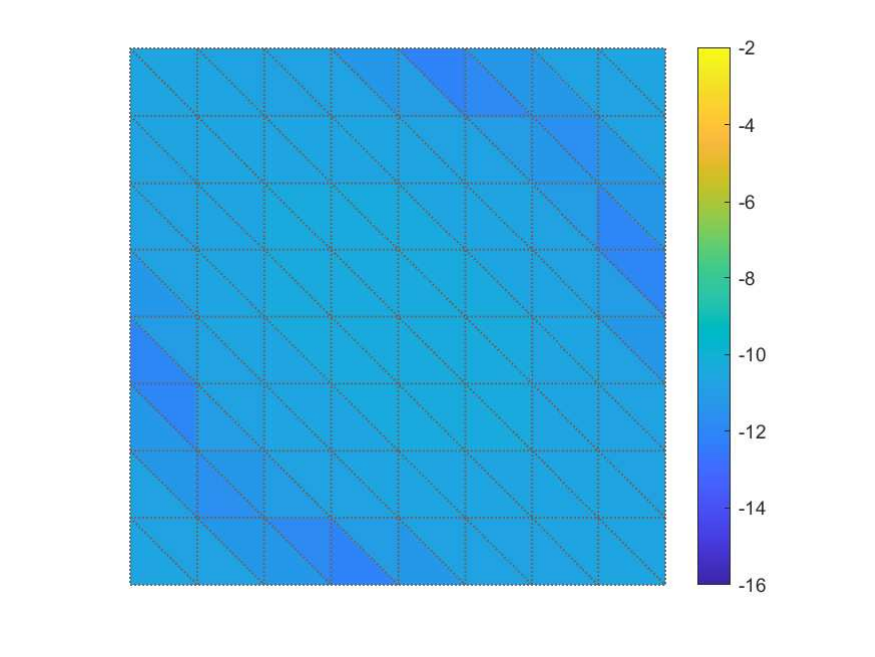}
    \end{minipage}}
    \subfigure{
    \begin{minipage}[t]{.20\textwidth}
    \centering
      \includegraphics[width=110pt]{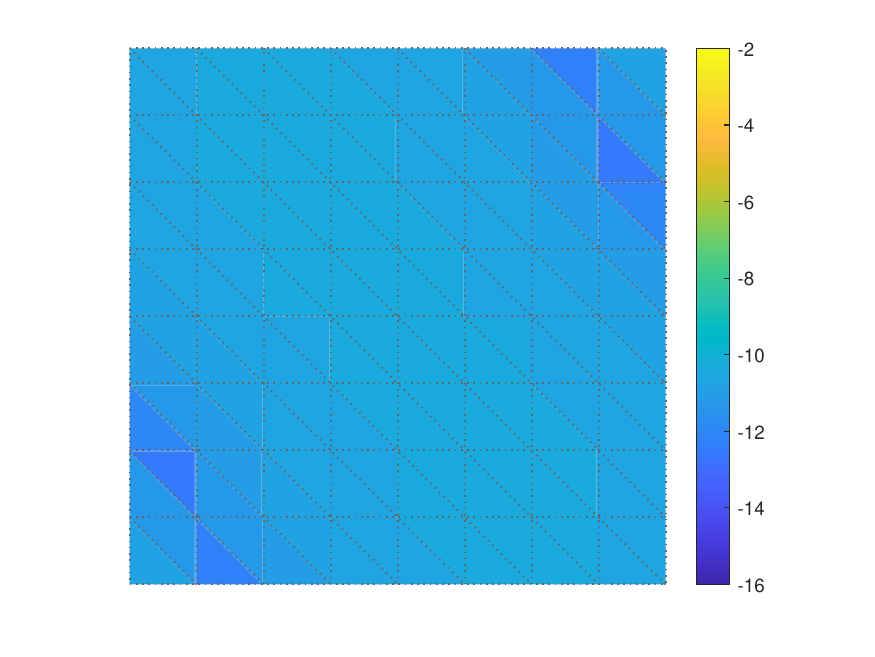}
    \end{minipage}}
    \subfigure{
    \begin{minipage}[t]{.20\textwidth}
    \centering
      \includegraphics[width=110pt]{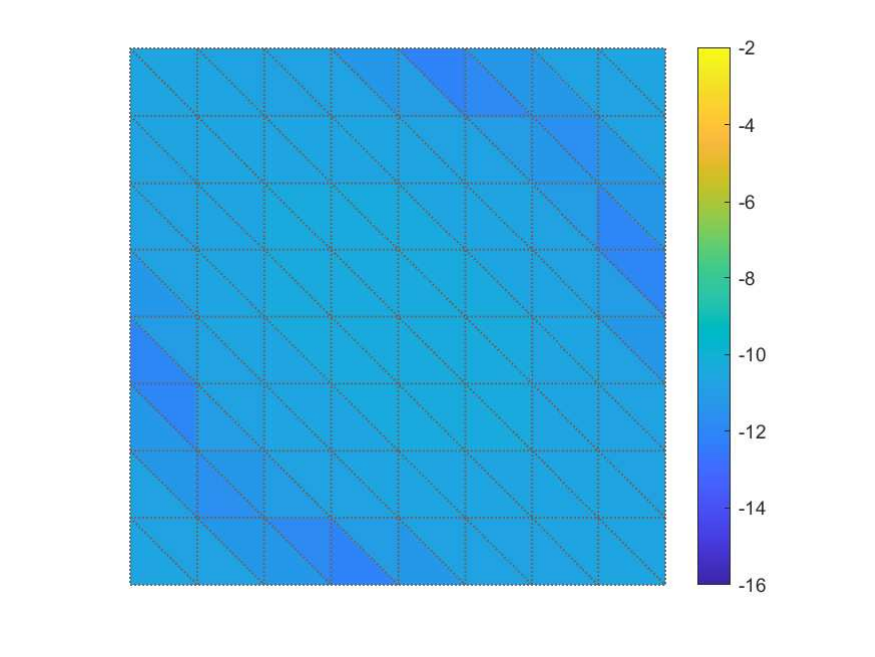}
    \end{minipage}}
    \subfigure{
    \begin{minipage}[t]{.20\textwidth}
    \centering
      \includegraphics[width=110pt]{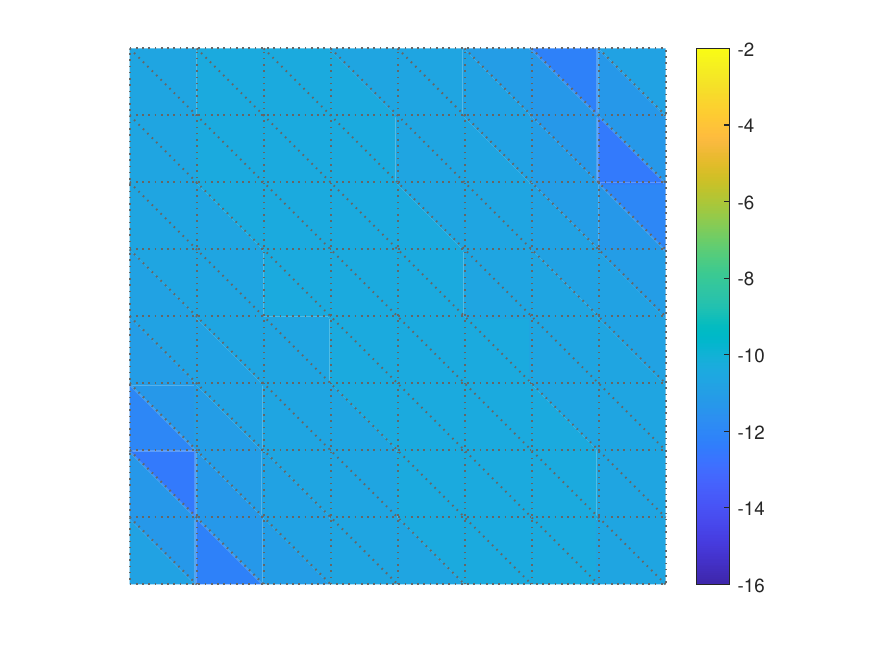}
    \end{minipage}}
    \qquad
    \subfigure{
    \rotatebox{90}{\scriptsize{~~~~~~~~~~~~Quartic}}
    \begin{minipage}[t]{.20\textwidth}
      \centering
      \includegraphics*[width=110pt]{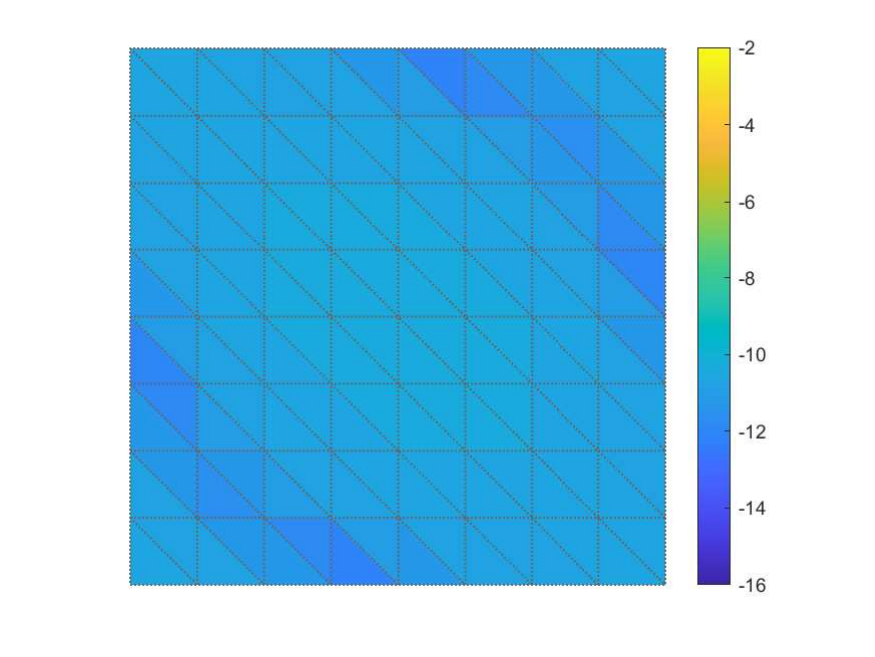}
    \end{minipage}}
    \subfigure{
    \begin{minipage}[t]{.20\textwidth}
    \centering
      \includegraphics[width=110pt]{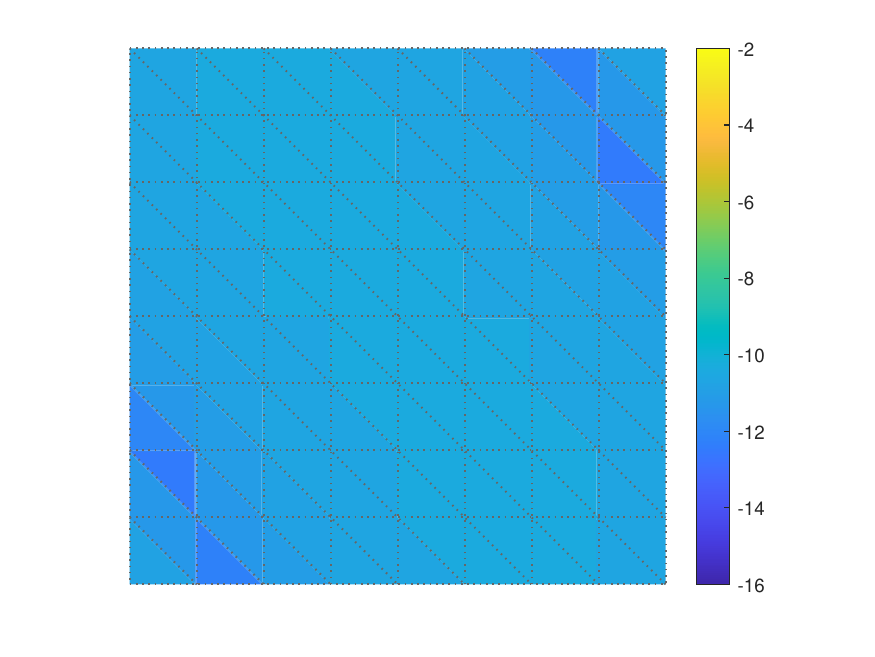}
    \end{minipage}}
    \subfigure{
    \begin{minipage}[t]{.20\textwidth}
    \centering
      \includegraphics[width=110pt]{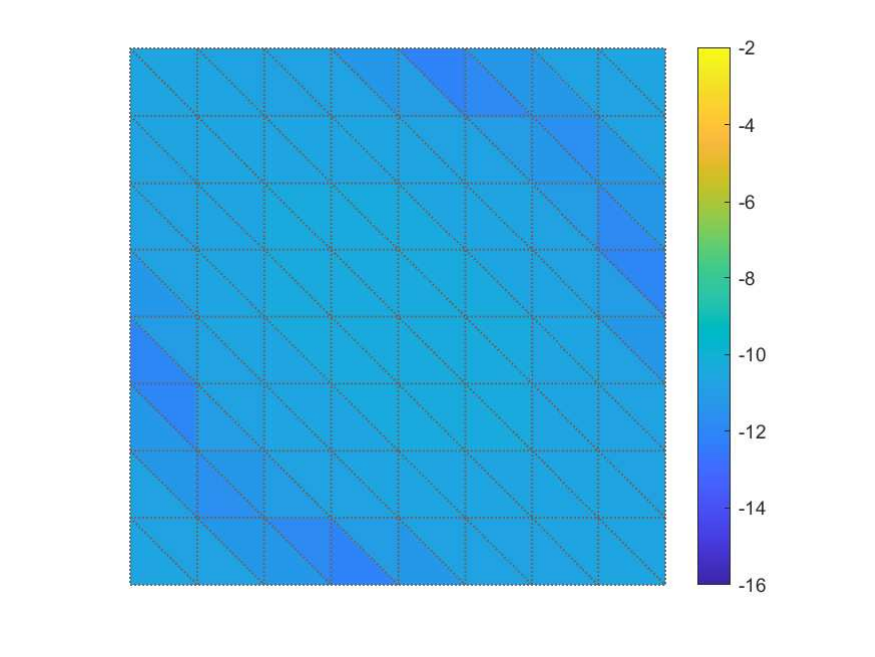}
    \end{minipage}}
    \subfigure{
    \begin{minipage}[t]{.20\textwidth}
    \centering
      \includegraphics[width=110pt]{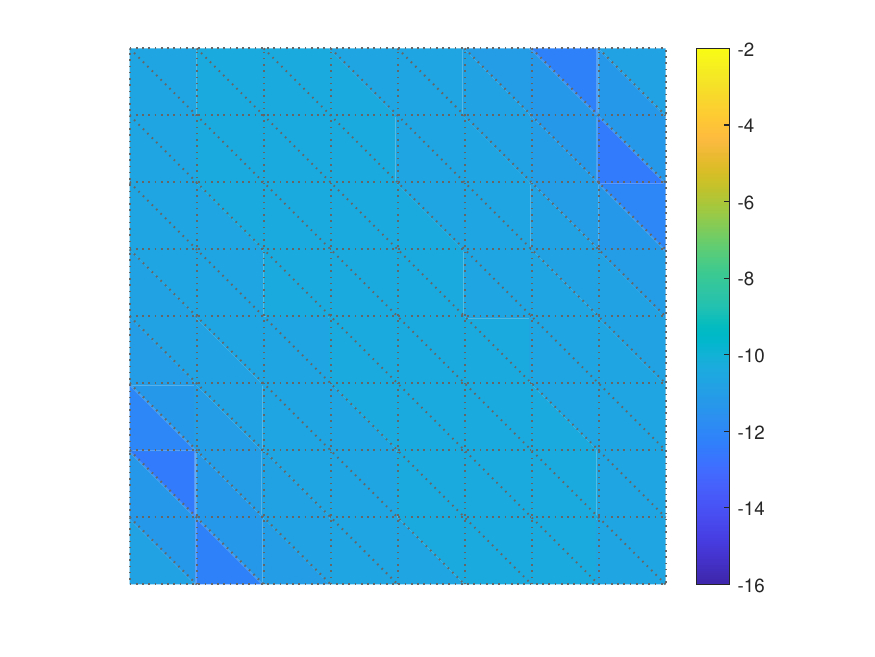}
    \end{minipage}}
    \caption{Conservation properties on $\mathcal{T}_{II}^{*}$ of the FVE-2L schemes for Example~\ref{ex:Elasticity problem}.
    The dual elements marked blue indicate that the local conservation is maintained. A mesh with
   $h \approx 1/8$ is used.}
    \label{fig:Conservation_error_ex2}
\end{figure}

\begin{figure}[!htbp]
    \centering
    \subfigure{
    \begin{minipage}[t]{.47\textwidth}
      \centering
      {(a) The $H^1$ convergence rate}\\
      \includegraphics*[width=180pt]{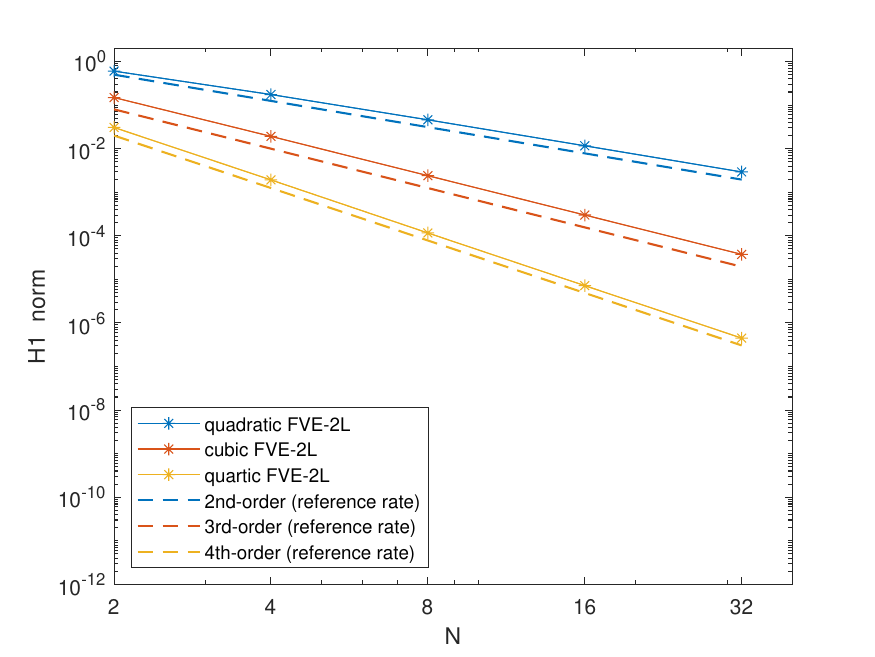}
    \end{minipage}}
    \subfigure{
    \begin{minipage}[t]{.47\textwidth}
    \centering
    {(b) The $L^2$ convergence rate}\\
      \includegraphics[width=180pt]{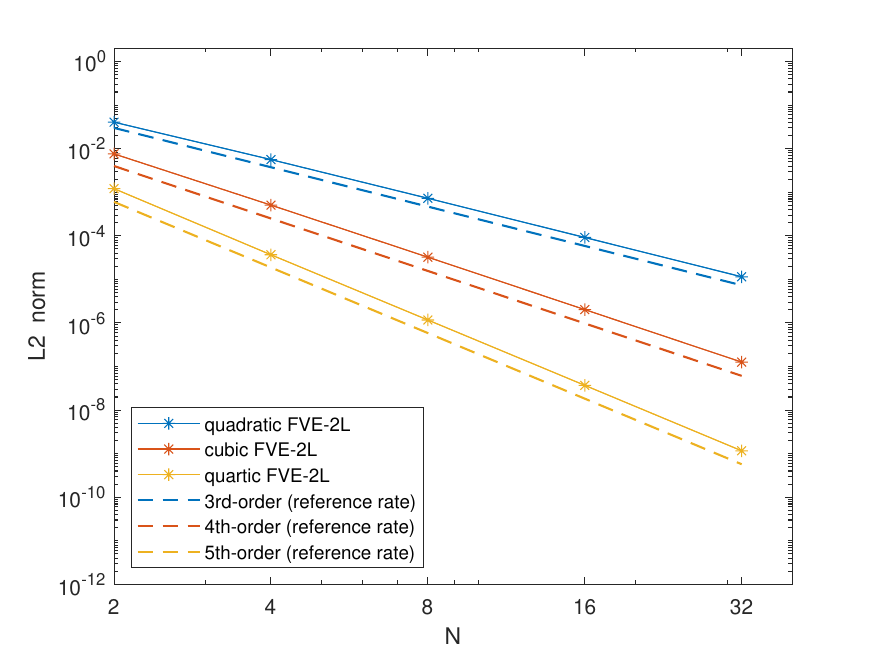}
    \end{minipage}}
    \caption{The numerical convergence rates of the FVE-2L schemes for Example~\ref{ex:Elasticity problem}.
    Here, $N$ denotes the number of intervals in each axis direction of the primary mesh.}
    \label{fig:The numerical results of k-order FVE-2L schemes for ex2}
\end{figure}

\begin{table}[htbp!]
\centering
  \caption{The global conservation law for Example~\ref{ex:Elasticity problem}}
  \label{tab:global conservation law for ex2}
  \begin{tabular}{ c | c  c c c  }
  \toprule
  FVE-2L scheme & $C_{\mathcal{T}_{II,flux}^{*}}^{1}$ & $C_{\mathcal{T}_{II,flux}^{*}}^2$ & $C_{\mathcal{T}_{II,equa}^{*}}^1$ & $C_{\mathcal{T}_{II,equa}^{*}}^2$  \\
  \hline
  quadratic & -2.7934e-09 & -3.7253e-09 & -2.7940e-09 & -3.7261e-09\\
  cubic  & -2.7986e-09 & -3.7303e-09 & -2.7994e-09 & -3.7314e-09\\
  quartic & -2.7977e-09 & -3.7281e-09 & -2.7984e-09 & -3.7293e-09\\
  \bottomrule
  \end{tabular}
\end{table}

We take $\lambda=1$ and $\mu=2$ and choose $\textbf{f}=(f_1,f_2)^{T}$ such that the exact solutions are $\mathbf{u}=(u_1,\,u_2)^{T}=(\sin(\pi x)\sin(\pi y),\,16x(x-1)y(y-1))^{T}$.
Here, we focus on the conservation properties on the second dual layer $\mathcal{T}_{II}^{*}$. The local conservation errors on each $K_{II}^{*}\in\mathcal{T}_{II}^{*}$ in flux and equation forms for (\ref{eq:Elasticity_equation}) can be written as
\begin{align}
&\textbf{C}_{K_{II}^{*},flux}=-\int_{\partial K_{II}^{*}}\sigma(\textbf{u}_h)\,\vec{n}\,\ud s-\iint_{K_{II}^{*}}\textbf{f}\,\ud x \ud y,\label{eq:elasticity_conservation1}\\
&\textbf{C}_{K_{II}^{*},equa}=-\iint_{K_{II}^{*}}\nabla \cdot \sigma(\textbf{u}_h)\,\ud x\ud y-\iint_{K_{II}^{*}}\textbf{f}\,\ud x\ud y,\label{eq:elasticity_conservation2}
\end{align}
both of which are vectors with two components, i.e., $\textbf{C}_{K_{II}^{*},flux}:=(C_{K_{II}^{*},flux}^{1},C_{K_{II}^{*},flux}^{2})^{T}$ and $\textbf{C}_{K_{II}^{*},equa}:=(C_{K_{II}^{*},equa}^{1},C_{K_{II}^{*},equa}^{2})^{T}$. Accordingly, the global conservation errors can be defined as
\begin{align*}
\textbf{C}_{\mathcal{T}_{II}^{*},flux} &= (C_{\mathcal{T}_{II}^{*},flux}^{1},C_{\mathcal{T}_{II}^{*},flux}^{2})^{T}
= \sum\limits_{K_{II}^{*}\in\mathcal{T}_{II}^{*}}\textbf{C}_{K_{II}^{*},flux},   \\
\textbf{C}_{\mathcal{T}_{II}^{*},equa} &= (C_{\mathcal{T}_{II}^{*},equa}^{1},C_{\mathcal{T}_{II}^{*},equa}^{2})^{T}
= \sum\limits_{K_{II}^{*}\in\mathcal{T}_{II}^{*}}\textbf{C}_{K_{II}^{*},equa}.
\end{align*}

Fig.~\ref{fig:Conservation_error_ex2} and Table~\ref{tab:global conservation law for ex2} show that both the local and global conservation law in flux and equation forms are maintained on the second layer $\mathcal{T}_{II}^{*}$ of the dual mesh by FVE-2L schemes.
Moreover, Fig.~\ref{fig:The numerical results of k-order FVE-2L schemes for ex2} shows the optimal $H^1$ and $L^2$ convergence rates
of the FVE-2L schemes, which is consistent with the theoretical analysis.

\end{example}

\section{Conclusions}
\label{sec:conclusion}
In the previous sections we have presented a family of high-order finite volume element schemes, FVE-2L schemes,
based on the two-layer dual mesh construction. The dual mesh consists of the barycenter dual mesh of the linear finite
volume element scheme (the first dual layer) and the triangulation of the primary mesh (the second dual layer).
This two-layer strategy provides a much simpler way to construct dual mesh elements and thus high-order FVE schemes
than the single-layer strategy used in the existing high-order FVE schemes.
Moreover, the FVE-2L schemes can avoid the effect of Dirichlet boundary conditions
and preserve the conservation law in both flux and equation forms; see (\ref{eq:conservation-law-flux-for-FVE-2L})
and (\ref{eq:conservation-law-equa-for-FVE-2L}). Furthermore, we have shown that the optimal regularity
for  the $L^2$ convergence of the $k$th-order FVE-2L scheme is $u\in H^{k+1}$
(cf. Theorem~\ref{thm:L2 estimate}), which is consistent with the approximation theory and is weaker than
$u\in H^{k+2}$ required by the existing $k$th-order FVE schemes.
A key to the error analysis of the FVE-2L schemes is the introduction of the parametric trial-to-test mapping (\ref{eq:trial-to-test-mapping-def})
and the minmax optimization problem (\ref{minimum_problem}) that allows the numerical computation of  the lower bound of the minimum angle condition and leads to weaker sufficient conditions for the stability of the FVE-2L schemes. This approach can also be used for other FVE schemes.
 Finally, numerical experiments have been presented to demonstrate the conservation and convergence properties of the FVE-2L schemes.

In this work we have used triangular meshes. It is worth pointing out that the dual mesh construction and stability analysis
in this work  can be extended to quadrilateral meshes, higher-order schemes, and even some mixed schemes.

\section*{Acknowledgements}
X. Wang and X. Zhang were supported in part by the National Natural Science Foundation of China (No.12371396).

\appendix
\section{Analytical expressions of test basis functions}
\label{SEC:test-basis-functions}

The basis functions of the test space restricted on the reference element $\hat{K}$ satisfy
(\ref{eq:test function}) and (\ref{eq:test function4}). Their analytical expressions are given in this appendix.
For the \textbf{quadratic} ($k=2$) FVE-2L scheme,
\begin{align*}
\begin{array}{lll}
\hat{\psi}_1=\left\{
            \begin{array}{ll}
              1 - 2x - 2y & \textrm{on}\,\, Q_{1},\\
              0           & \textrm{otherwise},
            \end{array}
            \right.
            &
    \hat{\psi}_3=\left\{
            \begin{array}{ll}
              2x - 1 & \textrm{on}\,\, Q_{2},\\
              0           & \textrm{otherwise},
            \end{array}
            \right.
            &
    \hat{\psi}_5=\left\{
            \begin{array}{ll}
              2y - 1 & \textrm{on}\,\, Q_{3},\\
              0           & \textrm{otherwise},
            \end{array}
            \right.
    \\
[0.1ex]\\
\hat{\psi}_2=\left\{
            \begin{array}{ll}
              2x & \textrm{on}\,\, Q_{1},\\
              2 - 2x - 2y & \textrm{on}\,\, Q_{2},\\
              0           & \textrm{otherwise},
            \end{array}
            \right.
            &
\hat{\psi}_4=\left\{
            \begin{array}{ll}
              2y & \textrm{on}\,\, Q_{2},\\
              2x & \textrm{on}\,\, Q_{3},\\
              0           & \textrm{otherwise},
            \end{array}
            \right.
            &
\hat{\psi}_6=\left\{
            \begin{array}{ll}
              2 - 2x - 2y & \textrm{on}\,\, Q_{3},\\
              2y & \textrm{on}\,\, Q_{1},\\
              0           & \textrm{otherwise},
            \end{array}
            \right.
    \\
[0.1ex]\\
\hat{\psi}_7 =1\quad \textrm{on}\,\, Q_{4}.
\end{array}
\end{align*}
For the \textbf{cubic} $(k=3)$ FVE-2L scheme,
\begin{align*}
\begin{array}{lll}
\hat{\psi}_1=\left\{
            \begin{array}{ll}
              1 - 3x - 3y & \textrm{on}\,\, Q_{1},\\
              0           & \textrm{otherwise},
            \end{array}
            \right.
            &
    \hat{\psi}_2=\left\{
            \begin{array}{ll}
              3x & \textrm{on}\,\, Q_{1},\\
              0           & \textrm{otherwise},
            \end{array}
            \right.
            &
    \hat{\psi}_9=\left\{
            \begin{array}{ll}
              3y & \textrm{on}\,\, Q_{1},\\
              0           & \textrm{otherwise},
            \end{array}
            \right.
    \\
[-0.25ex]\\
\hat{\psi}_3=\left\{
            \begin{array}{ll}
              3 - 3x - 3y & \textrm{on}\,\, Q_{2},\\
              0           & \textrm{otherwise},
            \end{array}
            \right.
            &
    \hat{\psi}_4=\left\{
            \begin{array}{ll}
              3x  - 2 & \textrm{on}\,\, Q_{2},\\
              0           & \textrm{otherwise},
            \end{array}
            \right.
            &
    \hat{\psi}_5=\left\{
            \begin{array}{ll}
              3y & \textrm{on}\,\, Q_{2},\\
              0           & \textrm{otherwise},
            \end{array}
            \right.
    \\
[-0.25ex]\\
\hat{\psi}_6=\left\{
            \begin{array}{ll}
              3x & \textrm{on}\,\, Q_{3},\\
              0           & \textrm{otherwise},
            \end{array}
            \right.
            &
    \hat{\psi}_7=\left\{
            \begin{array}{ll}
              3y  - 2 & \textrm{on}\,\, Q_{3},\\
              0           & \textrm{otherwise},
            \end{array}
            \right.
            &
    \hat{\psi}_8=\left\{
            \begin{array}{ll}
              3 - 3x - 3y & \textrm{on}\,\, Q_{3},\\
              0           & \textrm{otherwise},
            \end{array}
            \right.
    \\
[-0.25ex]\\
\hat{\psi}_{10} =1 \quad \textrm{on}\,\, Q_{4}.
\end{array}
\end{align*}
For the \textbf{quartic} ($k=4$) FVE-2L scheme,
\begin{align*}
\begin{array}{ll}
\hat{\psi}_1=\left\{
            \begin{array}{ll}
              1- 6x  -6y + 8x^2 +16xy +8y^2  & \textrm{on}\,\, Q_{1},\\
              0           & \textrm{otherwise},
            \end{array}
            \right.
            &
    \hat{\psi}_2=\left\{
            \begin{array}{ll}
              8x  -16 x^2 - 16xy & \textrm{on}\,\, Q_{1},\\
              0           & \textrm{otherwise},
            \end{array}
            \right.
    \\
[-0.25ex]\\
\hat{\psi}_{12}=\left\{
            \begin{array}{ll}
              8y  -16xy - 16y^2   & \textrm{on}\,\, Q_{1},\\
              0           & \textrm{otherwise},
            \end{array}
            \right.
            &
    \hat{\psi}_{4}=\left\{
            \begin{array}{ll}
              -8  +24x  +8y  -16x^2  -16xy & \textrm{on}\,\, Q_{2},\\
              0           & \textrm{otherwise},
            \end{array}
            \right.
    \\
[-0.25ex]\\
\hat{\psi}_{5}=\left\{
            \begin{array}{ll}
              3  -10x  +8x^2 & \textrm{on}\,\, Q_{2},\\
              0           & \textrm{otherwise},
            \end{array}
            \right.
            &
    \hat{\psi}_{6}=\left\{
            \begin{array}{ll}
              -8y  + 16xy & \textrm{on}\,\, Q_{2},\\
              0           & \textrm{otherwise},
            \end{array}
            \right.
    \\
[-0.25ex]\\
\hat{\psi}_{8}=\left\{
            \begin{array}{ll}
              -8x - 16xy  & \textrm{on}\,\, Q_{3},\\
              0           & \textrm{otherwise},
            \end{array}
            \right.
            &
    \hat{\psi}_{9}=\left\{
            \begin{array}{ll}
              3 - 10y +8y^2 & \textrm{on}\,\, Q_{3},\\
              0           & \textrm{otherwise},
            \end{array}
            \right.
    \\
[-0.25ex]\\
\hat{\psi}_{10}=\left\{
            \begin{array}{ll}
            -8 	-8x  +24y -16xy -16y^2   & \textrm{on}\,\, Q_{3},\\
              0           & \textrm{otherwise},
            \end{array}
            \right.
            &
    \hat{\psi}_{3}=\left\{
            \begin{array}{ll}
              -2x  +8 x^2 +8xy & \textrm{on}\,\, Q_{1},\\
              6  -14x  -6y  +8x^2  +8xy & \textrm{on}\,\, Q_{2},\\
              0           & \textrm{otherwise},
            \end{array}
            \right. 
    \\
[-0.25ex]\\
\hat{\psi}_{7}=\left\{
            \begin{array}{ll}
              6y  -8xy  & \textrm{on}\,\, Q_{2},\\
              6x  -8xy & \textrm{on}\,\, Q_{3},\\
              0           & \textrm{otherwise},
            \end{array}
            \right.
            &
    \hat{\psi}_{11}=\left\{
            \begin{array}{ll}
              -2y  +8xy  +8y^2  & \textrm{on}\,\, Q_{1},\\
              6  -6x  -14y  +8xy   +8y^2   & \textrm{on}\,\, Q_{3},\\
              0           & \textrm{otherwise},
            \end{array}
            \right.
    \\
[-0.25ex]\\
\hat{\psi}_{13} =3 - 4x - 4y \quad \textrm{on}\,\, Q_{4},
        &
        \hat{\psi}_{14} =- 1  +4x  \quad \textrm{on}\,\, Q_{4},
    \\
[-0.25ex]\\
\hat{\psi}_{15}=- 1  + 4y  \quad \textrm{on}\,\, Q_{4}.
\end{array}
\end{align*}

\section{Trial-to-test mapping} 
\label{Appendix:trial-to-test-mapping}

The parametric trial-to-test mapping for FVE-2L schemes ($k$=$2,\,3,\,4$) is defined in (\ref{eq:trial-to-test-mapping-def}),
where \textbf{the parametric transformation matrix} $M_k(\textbf{a},\textbf{b})$ is defined as
\begin{align}\label{eq:M2}
M_2(\textbf{a},\textbf{b})&:=
\left(
\setlength{\arraycolsep}{2.0pt}
\begin{array}{ccccccc}
1 & & & & & &\\
a_1 & a_2 & a_1 & & &&\\
& & 1 & & & &\\
& & a_1 & a_2 & a_1 &&\\
& & & & 1 & &\\
a_1 & & && a_1 &a_2\\
b_1&b_2&b_1&b_2&b_1&b_2&b_3
\end{array}
\right),\\
M_3(\textbf{a},\textbf{b})&:=
\left(
\setlength{\arraycolsep}{2.0pt}
\begin{array}{cccccccccc}
1 & & & & & & & & &\\
a_1& a_2 &a_3 &a_4 & & & & & &\\
a_4& a_3& a_2 &a_1 & & & & & &\\
& & & 1 & & & & & &\\
& & & a_1& a_2 & a_3&a_4 & & &\\
& & & a_4& a_3& a_2 &a_1 & & &\\
& & & & & & 1 & & &\\
a_4& & & & & & a_1& a_2 & a_3&\\
a_1& & & & & & a_4& a_3& a_2 &\\
b_1&b_2&b_2&b_1&b_2&b_2&b_1&b_2&b_2&b_3
\end{array}
\right),    \label{eq:M3}   \\
M_4(\textbf{a},\textbf{b})&:=
\left(
\setlength{\arraycolsep}{2.0pt}
\begin{array}{ccccccccccccccc}
1 & & & & & & & & & & & & & &\\
a_1& a_2 &a_3 &a_4 &a_5 & & & & & & & & & &\\
a_6&a_7 & a_8 & a_7& a_6& & & & & & & & & &\\
a_5& a_4& a_3& a_2 & a_1& & & & & & & & & &\\
& & & & 1 & & & & & & & & & &\\
& & & & a_1& a_2 &a_3 &a_4 &a_5 & & & & &\\
& & & & a_6&a_7 & a_8 & a_7& a_6 & & & & &\\
& & & & a_5& a_4& a_3& a_2 & a_1 & & & & &\\
& & & & & & & & 1 & & & & & &\\
a_5& & & & & & & & a_1& a_2 &a_3 &a_4 && &\\
a_6& & & & & & & & a_6&a_7 & a_8 & a_7& & &\\
a_1& & & & & & & & a_5& a_4& a_3& a_2 & & &\\
b_1 & b_2&b_3 &b_4 &b_5 &b_6 &b_7 &b_6 &b_5 &b_4 &b_3 &b_2 &b_8 &b_9 &b_9\\
b_5& b_4&b_3 &b_2 &b_1  &b_2 &b_3 &b_4 &b_5 &b_6 &b_7 &b_6 & b_9 &b_8 &b_9\\
b_5& b_6&b_7 &b_6 &b_5 &b_4 &b_3 &b_2 &b_1  &b_2 &b_3 &b_4 & b_9& b_9&b_8\\
\end{array}
\right).\label{eq:M4}
\end{align}
Here, $\V{a}$ and $\V{b}$ are the parameters. Their positions in the matrices are based on the symmetry and location of
the interpolation nodes (for the trial space) with respect to the location associated with the test function.
We also require that the trial-to-test mapping reproduce the uniform solutions, i.e., $\hat{u}_h \equiv 1$ implies
$\hat{v}_{h,\mathrm{I}}\equiv 1$ and $\hat{v}_{h,\mathrm{II}}=0$. This implies that the row sums of $M_{k}(\textbf{a},\textbf{b})$ corresponding to the first dual layer
equal to 1 and the row sums corresponding to the second dual layer equal to 0, i.e., 
\begin{align}
\label{eq:restictions_pih}
\begin{array}{ll}
\left\{
            \begin{array}{l}
              a_2=1-2a_1,\\
              b_3=-3b_1-3b_2,
            \end{array}
\right. &\textrm{for}\,\,k=2 ,
    \\
[-0.25ex]\\
\left\{
            \begin{array}{l}
              a_2=1-a_1-a_3-a_4,\\
              b_3=-3b_1-6b_2,
            \end{array}
\right. &\textrm{for}\,\, k=3,
            \\
[-0.25ex]\\
\left\{
            \begin{array}{l}
              a_2=1-a_1-a_3-a_4-a_5,\\
              a_8=1-2a_6-2a_7,\\
              b_8=-b_1-2b_2-2b_3-2b_4-2b_5-2b_6-b_7-2b_9,
            \end{array}
\right. &\textrm{for}\,\, k=4.
    \\
\end{array}
\end{align}

\end{document}